\documentclass[12t,a4paper]{amsart}
\usepackage{amsmath,amssymb,amsfonts, float}
\usepackage[all]{xy}
\usepackage{caption}
\usepackage{enumerate}
\usepackage{mathpazo}
\usepackage{a4wide}
\usepackage[normalem]{ulem}

\usepackage{bm}% bold math
\usepackage{graphicx}
\usepackage[breaklinks=true]{hyperref}

\usepackage{xcolor}
\usepackage{multicol}
\usepackage{tabularx}
\usepackage{comment}

\usepackage{hyperref}
\usepackage[capitalize]{cleveref}

\usepackage[normalem]{ulem}

\newcommand{\tung}[1]{{\color{red}#1}}

\newtheorem{thm}{Theorem}[section] 
\newtheorem*{thm*}{Theorem} 
\newtheorem{prop}[thm]{Proposition}
\newtheorem{lem}[thm]{Lemma}
\newtheorem{cor}[thm]{Corollary}

\theoremstyle{definition}
\newtheorem{definition}[thm]{Definition}
\newtheorem{expl}[thm]{Example}

\newtheorem{rem}[thm]{Remark}

\DeclareMathOperator{\C}{\mathbb{C}}
\DeclareMathOperator{\Z}{\mathbb{Z}}
\DeclareMathOperator{\N}{\mathbb{N}}

\DeclareMathOperator{\Q}{\mathbb{Q}}

\DeclareMathOperator{\Gal}{\text{Gal}}

\DeclareMathOperator{\supp}{{\rm Supp}}
\DeclareMathOperator{\Div}{\text{Div}}

\DeclareMathOperator{\Supp}{\text{Supp}}

\newcommand{\Ann}{{\rm Ann}}
\newcommand{\odd}{{\rm odd}}

    \DeclareFontFamily{U}{wncy}{}
    \DeclareFontShape{U}{wncy}{m}{n}{<->wncyr10}{}
    \DeclareSymbolFont{mcy}{U}{wncy}{m}{n}
    \DeclareMathSymbol{\Sha}{\mathord}{mcy}{"58}

\numberwithin{equation}{section}

\DeclareSymbolFont{bbold}{U}{bbold}{m}{n}
\DeclareSymbolFontAlphabet{\mathbbold}{bbold}

\begin{document}
\title{Cyclotomic factors of rational necklace functions}
\author{Nguyen Cao Minh, Nguyen Vu Hoang Minh, Dung Nguyen \\ Tung T. Nguyen,  Nguyen Duy Tan, Duong Tran}

\address{Department of Mathematics and Informatics, Ho Chi Minh City University of Education, Ho Chi Minh City, Vietnam}
\email{nguyencaominhtanthoi@gmail.com}

\address{Faculty of Mathematics and Computer Science, Ho Chi Minh City University of Science, Ho Chi Minh City, Vietnam}
\email{nvhminh2004@gmail.com}

 \address{Department of Mathematics, Elmhurst University, Elmhurst, Illinois, USA}
 \email{dnguy9448@365.elmhurst.edu}

 \address{Department of Mathematics, Elmhurst University, Elmhurst, Illinois, USA}
 \email{tung.nguyen@elmhurst.edu}
 
  \address{Faculty of Mathematics and Informatics, Hanoi University of Science and Technology, 1 Dai Co Viet Road, Hanoi, Vietnam } 
\email{tan.nguyenduy@hust.edu.vn}

\address{Michigan State University, USA}
\email{tdkduonggg@gmail.com}

\thanks{Our group is supported by the Vietnam Institute for Advanced Study in Mathematics (VIASM) through the REU 2025 (Research Experience for Undergraduates) program. TTN is partially supported by an AMS-Simons Travel Grant. NDT is partially supported by the Vietnam National
Foundation for Science and Technology Development (NAFOSTED) under grant number 101.04-2023.21}

\begin{abstract}
Necklace polynomials arise in various fields of mathematics, including combinatorics, Lie theory, and Galois theory over finite fields. Their arithmetic properties have been extensively studied in recent years. In this article, we introduce a new class of rational necklace functions that unifies two well-studied families of polynomials: necklace polynomials and Fekete polynomials. We describe several ways in which cyclotomic polynomials appear as factors of these rational necklace functions. Our results unify and generalize various earlier work  on necklace polynomials and on Fekete polynomials. In particular, we describe a surprising phenomenon in which certain Galois groups play a hidden role in the appearance of new cyclotomic factors that are not covered by these previous works.

\end{abstract}

\maketitle
\tableofcontents

\section{Introduction}
The $n$-th necklace polynomial $M_n$ is defined as 
\[ M_n(x)=\frac{1}{n} \sum_{d \mid n} \mu(d)x^{n/{d}},\]
where $\mu$ is the Möbius function. These polynomials arise in multiple areas of mathematics. In combinatorics, they count the number of distinct necklaces with $n$ beads, each colored from $x$ available colors, up to rotational symmetry (see \cite{moreau1872permutations}). In field theory, necklace polynomials give the number of monic irreducible polynomials of degree $n$ over a finite field with $x$ elements (see \cite{chebolu2011counting, ireland1990classical} for different approaches to this statement). In algebra, they describe the dimension of the degree-$n$ component of a free Jordan algebra (see \cite{metropolis1983witt}).

While extensive literature explores the combinatorial aspects of these polynomials, their arithmetic properties are less studied. Hyde initiated a systematic study of their factorization in \cite{hyde2022cyclotomic}. There, he analyzed cyclotomic factors of these polynomials using character theory and ring-theoretic properties of the so-called Mahler algebras. A similar approach is also used in \cite{Dynatomic2022} to study factors of generalized dynatomic polynomials.

Our interest in necklace polynomials stems from a seemingly  unrelated class of polynomials. More precisely, in \cite{chidambaram2023fekete}, a subset of the authors study the arithmetic of the $n$-th Fekete polynomials defined as
\begin{align*}
\label{def:Fekete}
    F_n(x)= \sum_{\substack{0 \leq a \leq n-1\\ \gcd(a,n)=1}} x^a.
\end{align*}
These polynomials are a special case of a wider class of polynomials associated with Dirichlet characters and they appear naturally in the theory of $L$-functions (see \cite{baker1990oscillations, conrey2000zeros, minavc2023arithmetic}). It is also interesting to remark that the values of $F_n$ at roots of unity have an elegant explicit description by the theory of Ramanujan sums (see \cite{fowler2014ramanujan, ramanujan1918certain}). Furthermore, these sums are directly related to the spectra of certain associated gcd-graphs (see \cite{klotz2007some, nguyen2026gcd}). While the Fekete polynomial $F_n(x)$ looks quite different from the necklace polynomial $M_n(x)$, it turns out that up to a minor modification,  they both satisfy an almost identical recursive formula. More precisely, we show in \cite[Proposition 2.17]{chidambaram2023fekete} that if $n$ is square-free, then
\[ \tilde{F}_n(x)=\sum_{d\mid n}\mu(d)\dfrac{x^d}{1-x^d} = \mu(n)\sum_{d\mid n}\mu(d)\dfrac{x^{{n}/{d}}}{1-x^{n/d}} , \]
where $\tilde{F}_n(x)= \dfrac{F_n(x)}{1-x^n}$. In \cite{chidambaram2023fekete}, using a combinatorial argument, we find various cyclotomic factors of $F_n$ (see \cite[Theorem 2.19]{chidambaram2023fekete}). Our argument, while does not use the Mahler algebra and appears to be different from Hyde's, is similar in spirit. Although \cite[Theorem 2.19]{chidambaram2023fekete} explains most of the cyclotomic factors of $F_n$, our numerical data shows that some cannot be explained by it.

In this article, we introduce a theory of rational necklace functions, which provides a general framework for both necklace and Fekete polynomials. More precisely, if $f \in \Q(x)$ is a rational function, we define the $n$-th necklace rational function of $f$ as 
\[ M_n(f) = \sum_{d|n} \mu(d) f(x^d) = \psi_n(f), \]
Here $\psi_n = \sum_{d \mid n} \mu(d)d$ is an element in the Mahler algebra $\Z[\Z^{\circ}]$ that we will formally introduce later. We remark that, for simplicity, we normalize $M_n(f)$; this will not affect our study of its factors. With this normalization, $\tilde{F}_n(x)$ is nothing but $M_n\left(\dfrac{x}{1-x} \right)$. Similarly, if $n$ is squarefree, then up to a factor of $\mu(n)n$, our definition recovers the definition of the classical necklace polynomial when $f=x.$

In general, $M_n(f)$ is a rational function, so we must be precise about what we mean by a factor. We say that $\Phi_d$ is a factor of a rational function $F \in \Q(x)$ if $ F=\frac{P}{Q}$ where  $P,Q\in\mathbb{Q}[x],\ \gcd(P,Q)=1$
and $\Phi_d\mid P$. Equivalently—since $\Phi_d$ is irreducible in $\Q(x)$—$\Phi_d$ is a factor of $F$ if and only if $F(\zeta_d)=0$.

Utilizing this general framework, we first explain a theorem that not only generalizes but also offers a conceptual explanation for the main results in \cite{chidambaram2023fekete, hyde2022cyclotomic} about cyclotomic factors of necklace and Fekete polynomials. More precisely, we show that various cyclotomic factors of a rational necklace functions could come from two sources: the \textit{congruence property} of the element $\psi_n$ described above, the   \textit{extra symmetries} of the original rational function $f$; and their interactions. We remark that while the second property does not hold for classical necklace polynomials, it does hold for the class of Fekete polynomials. More precisely, the rational function $f=\dfrac{x}{1-x}$ associated with Fekete polynomials has the property that 
\[ f(x)+f(1/x)=-1. \] 
This observation is exploited in \cite[Theorem 2.19]{chidambaram2023fekete}) and it provides a conceptual explanation for why, empirically, Fekete polynomials often possess more cyclotomic factors than the classical necklace polynomials. Second, using the Mahler algebras, we find a completely new source of cyclotomic factors for $M_n(f)$ which are  missing from both \cite{chidambaram2023fekete,hyde2022cyclotomic}. More precisely, we discover a \textit{lowering field of definition} phenomenon which arises from the Galois group of $\Q(\zeta_d)$ and its interaction with $\psi_n$ and the symmetries of $f.$ In this article, we describe precisely when this phenomenon happens.

Our approach is inspired by the strategy outlined in \cite{hyde2019polynomial, hyde2022cyclotomic} and \cite{Dynatomic2022}, which leverages ring-theoretic properties of certain Mahler algebras. In addition to using these algebras to study cyclotomic factors of rational necklace functions, we describe further properties of the algebras and their modules that may be of independent interest. We hope this approach will open new avenues for studying other necklace-like functions, such as the dynatomic polynomials of \cite{Dynatomic2022}.

{We state here two prototypical theorems concerning the relationships among the symmetries of $f$, the congruence properties of $\psi_n$, and the cyclotomic factors of $M_n(f)$. For notation and additional results, we refer the reader to the main text.

\begin{thm} (See \cref{prop:0mod-rational}) \label{intro1}
    Let $\alpha\in  \Z[\Z^\circ]$, $d>1$ and $f(x)$ a rational function. Suppose that $\phi_d(\alpha)=0$ and $f(x)$ is defined at $\zeta_d^m$, for every $m \in \supp(\alpha)$. Then $\Phi_d$ is a factor of $\alpha f$.
\end{thm}

\begin{thm} (See \cref{thm:mahler-congruence_revised})
   % Let $I\subset R=\Z[\Z^{\circ}]$ be an ideal and  $f\in \Q(x)$ with $I\subset  \Ann_{R}(f)$. Let $\alpha \in \Z[\Z^{\circ}]$.
   Let $f\in \Q(x)$ and $\alpha\in R=\Z[\Z^{\circ}]$.
    Suppose there exists $\beta\in \Ann_R(f)$ such that $\phi_d(\alpha)=\phi_d(\beta)$ and $f$ is defined at $\zeta_d^a$ for each $a \in \supp(\alpha)\cup \supp(\beta)$. Then $\Phi_d$ is a factor of $\alpha f$. 
\end{thm}

Here is a brief outline of our article. 

\subsection{Outline.} 
In \cref{sec:mahler}, we introduce the Mahler algebra $\Z[\Z^{\circ}]$ and the generalized necklace rational functions $M_n(f)$. We describe a theorem explaining how the symmetry of $f$ and a congruence property of $\psi_n$ give rise to certain cyclotomic factors of $M_n(f)$ (see \cref{thm:mahler-congruence_revised}). For certain symmetries, we also give an explicit condition for such a congruence of $\psi_n$ to occur.  In \cref{sec:signed-mahler}, we introduce a signed Mahler algebra to give a conceptual explanation for cyclotomic factors of $M_n(f)$ of the form $\Phi_{2m}$. In the spirit of the main results in \cref{sec:mahler}, we explain the relationship between the symmetries of $f$, the congruence property of $\psi_n$ in the signed Mahler algebra, and certain cyclotomic factors of $M_n(f)$ (see \cref{thm:mahler-congruence_signed}). Finally, in \cref{sec:galois}, we explain the lowering field of definition phenomenon. We explain how this new phenomenon provides new insight to cyclotomic factors of rational necklace functions. Finally, in the appendix, we give some ring-theoretic properties of $\Z[\Z^{\circ}]$ which may be of independent interest. In particular, we study its zero-divisors and their relation to rational functions with symmetries.

\section{Mahler algebra, symmetries, and cyclotomic factors of rational necklace functions} \label{sec:mahler}
In this section we discuss the appearance of various cyclotomic factors of $M_n(f)$ via the Mahler algebra introduced in \cite{hyde2019polynomial} (with a slight modification to deal with rational functions). In particular, we show that the presence of additional symmetries in $f$ contributes to the existence of several factors of $M_n(f)$.

\subsection{Mahler algebra}
Let $\Psi$ be the integral monoid ring $\Z[\N^\circ]$, where $\N^\circ$ is the multiplicative monoid of positive integers introduced in \cite{ hyde2019polynomial}. Explicitly, as a set, $\Psi$ consists of all integral linear combinations of formal expressions $[m]$ with $m\in \N$. Addition in $\Psi$ is defined  componentwise, and the multiplication in $\Psi$ is defined by extending the relation $[m][n]=[mn]$. Since we are dealing with rational functions and their symmetries, it is natural to extend this ring to include negative exponents as well. More precisely, let $R$ be the integral monoid ring $\Z[\Z^{\circ}]$, where $\Z^{\circ}$ is the multiplicative monoid of non-zero integers. The inclusion of negative indices allows us to encode symmetries involving $x \mapsto x^{-1}$, which will play an essential role in the later part of the article. Naturally, $\Z[\N^{\circ}]$ is a subring of $\Z[\Z^{\circ}]$.

Let $M= \Q(x)$, the field of all rational functions with coefficients in $\Q.$ Then, $M$ is a module over $R$ where the action of  $\alpha = \sum_{m \in \Z^{\circ}} a_m[m]$ on $f(x)$ is given by 

\[ (\alpha f)(x) = \sum_{m \in \Z^{\circ}} a_m ([m] \cdot f(x))= \sum_{m \in \Z^{\circ}} a_m f(x^m). \]

We note that if $\alpha\in \Psi$ satisfies that $\alpha x=0$ then $\alpha=0$. In fact, write $\alpha =\sum_m c_m[m]$, then $0=\alpha x=\sum_m c_m x^m$. Hence $c_m=0$, for all $m$, i.e., $\alpha=0$.

Define $\varphi_n:=\sum_{d\mid n}\mu(d) [n/d] \in \Psi$. Then $M_n(x)= \varphi_n x$. We also define $$\psi_n:=\sum_{d\mid n}\mu(d) [d] \in \Psi.$$
Note that if $n$ is square free then 
\[
\psi_n=\mu(n) \varphi_n = {\prod_{p \mid n} (1-[p]).}
\]

\begin{rem}
Throughout this article we use the normalization \(M_n(f)=\psi_n f\). By the above explanation, this differs from the classical normalization only by a nonzero scalar.
\end{rem}

\begin{rem}
In \cite{hyde2022cyclotomic}, Hyde also studies the cyclotomic factors of the shifted cyclotomic polynomials $\Phi_n(x)-1.$ The formalism that we describe in this article applies to this case as well. More precisely, similar to the \textit{additive} necklace rational functions $M_n(f)$, we can also define the notion of a \textit{multiplicative} necklace rational function. To do so, we observe that $\Q(x)^{\times}$ can also be considered as a module over $\Z[\Z^{\circ}]$ where the action of $\alpha:=\sum_{m \in \Z^{\circ}} a_m [m]\in \Z[\Z^{\circ}]$ on $f$ is given by 
\[ \alpha*f  = \prod_{m \in \Z^{\circ}} (a_m[m])*f=  \prod_{m \in \Z^{\circ}} f(x^m)^{a_m}.\]
We then define $P_{n}(f) = \psi_n * f$. With this definition, $\Phi_n(x) = \left[P_n(x-1)\right]^{\mu(n)}$.  All theorems that we discussed in this article for the additive necklace rational function has a direct analog for the shifted multiplicative necklace rational function $P_n(f) -1$ using this module structure of $\Q(x)^{\times}$ over $\Z[\Z^{\circ}].$ Formally speaking, we can think about $P_n(f)-1$ as $\exp(\psi_n \log(f))-1$, even though rigorously speaking $\log(f)$ is not in $\Q(x)$.
\end{rem}

%By \cite[Proposition 4.2.1]{hyde2019polynomial}, the map $\alpha\mapsto \alpha x$ defines an isomorphism between $\Psi[m]$ and $\Z[x]/(x^m-1)$ as $\Psi$-modules.

Let $(\Z/d\Z)^\circ$ be the multiplicative monoid of the ring $\Z/d\Z$ (including the zero element). The map $[m]\mapsto [m \mod d]$ induces a ring homomorphism $\phi_d\colon \Z[\Z^\circ]\to \Z[(\Z/d\Z)^\circ]$. The image of an element $\alpha\in \Z[\Z^\circ]$ via this ring homomorphism is denoted by $\overline{\alpha}:=\phi_d(\alpha)$. Without introducing another notation, we also denote by $\phi_d$ to be the restriction of $\phi_d$ to $\Psi$; namely $\phi_d: \Z[\N^{\circ}] \to \Z[(\Z/d)^{\circ}].$ 

Let $\alpha,\beta\in \Z[\Z^\circ]$, we say that $\alpha$ is congruent to $\beta$ modulo $m$ inside bracket, and write $\alpha\equiv\beta \mod [m]$ if 
$\alpha-\beta$ can be written in the form 
\[\alpha-\beta= \sum_k n_k([a_k]-[b_k]), \]
where $n_k\in \Z$ and $a_k\equiv b_k \mod m$ for every $k$. One can check that this is an equivalence relation. % and $\Psi$-invariant.
The following statement is a slight generalization of \cite[Corollary 4.2.2, part 1]{hyde2019polynomial}

\begin{prop}
\label{prop:0modm}
    Let $\alpha\in \Z[\Z^\circ]$ and $d\geq 1$. The following conditions are equivalent.
    \begin{enumerate}
    \item $\overline{\alpha}=0$ in $\Z[(\Z/d\Z)^\circ]$.
        \item $\alpha\equiv 0\mod [d]$.
\end{enumerate}
Assume further that $\alpha\in \Z[\N^\circ]$. Then the above two conditions are also equivalent to 
\begin{enumerate}
        \item[(3)] $x^d-1$ divides $\alpha f(x)$ for all $f(x)\in \Z[x]$;
        \item[(4)] $x^d-1$ divides $\alpha x$.
    \end{enumerate}
\end{prop}
\begin{proof}
 $(1)\Rightarrow (2)$: Suppose that $\alpha =\sum\limits_k n_k[a_k]\in \Z[\Z^\circ]$ and $\overline{\alpha}=0$.
Then 
\[
\overline{\alpha}=\sum_{1\leq a\leq d} n_a [\overline{a}],
\]
 where $n_a=\sum\limits_{\substack{k:\, a_k\equiv a\mod d}} n_k$. Because $\overline{\alpha}=0$, $n_a=0$, for all $1\leq a\leq d$. We have
 \[
 \alpha= \sum\limits_k n_k[a_k] =\sum_{1\leq a\leq d }\sum\limits_{k:\,a_k\equiv a} n_k[a_k]=\sum_{1\leq a\leq d }\sum\limits_{k:\,a_k\equiv a} n_k([a_k]-[a]).
 \]
This implies that $\alpha\equiv 0\mod {[d]}$.

 $(2)\Rightarrow (1)$: Suppose that $\alpha\equiv 0\mod [d]$. Then $\alpha=\sum n_k([a_k]-[b_k])$, where $a_k\equiv b_k\mod d$ for every $k$. Clearly, $\bar{\alpha}=\phi_d(\alpha)=\sum n_k([\bar{a_k}]-[\bar{b_k}])=0$ in $\Z[(\Z/d\Z)^\circ]$.

Now we assume further that $\alpha\in \Z[\N^\circ]$.

\medskip
    $(2)\Rightarrow (3)$: Suppose that $\alpha\equiv 0\mod [d]$ and $f(x)$ is a polynomial in $\Z[x]$. We can write $\alpha =\sum_{k} n_k([a_k]-[b_k])$, where $n_k\in \Z$ and $a_k\equiv b_k\mod d$. We have
    \[
    \alpha f(x)=\sum_k n_k (f(x^{a_k})-f(x^{b_k})).
    \]
    Since $a_k\equiv b_k\mod d$, this implies that $f(x^{a_k})-f(x^{b_k})$ is divisible by $x^d-1$. Hence $x^d-1$ divides $\alpha f(x)$.

    \medskip
\noindent      $(3)\Rightarrow (4)$: Clearly, $(3)$ implies $(4)$.

\medskip
    \noindent $(4)\Rightarrow (1)$: Suppose that $x^d-1$ divides $\alpha x$. We can write
    \[
    \alpha x= (x^d-1)(\sum_{k=0}^n n_k x^{b_k}),
    \]
    for some $n_k\in \Z$, $b_k\in \N$. We have
    \[
     \alpha x =  \sum_{k=0}^n n_k (x^{b_k+d}-x^{b_k}) =\left( \sum_{k=0}^n n_k ([b_k+d]-[b_k])\right ) x.
    \]
    Hence $\alpha=\sum_{k=0}^n n_k ([b_k+d]-[b_k])$. Thus $\overline{\alpha}=0$.
\end{proof} 

\iffalse
\begin{rem} 
    In the proof of the implication (1) to (2), if we start with $\alpha =\sum\limits_k n_k[a_k]\in \Z[\N^\circ]$ such that $a_k\not\equiv 0\pmod d$ for $n_k\not=0$ and $\overline{\alpha}=0$, then 
     \[
 \alpha= \sum_{0< a<d }\sum\limits_{k:\,a_k\equiv a} n_k([a_k]-[a]).
 \]
\end{rem}
\fi

There is a rather straightforward generalization of \cref{prop:0modm} to any rational function $f \in \Q(x)$, with one caveat: we need to be careful about the poles of $f$ at roots of unity. Since this is a phenomenon that happens throughout the text, we introduce the following formal definition.

\begin{definition}
 Let $G$ be a multiplicative monoid.   For $\alpha  = \sum_{m} a_m [m] \in \Z[G]$, we define the support $\Supp(\alpha)$ of $\alpha$ to be the set $\{m \mid a_m \neq 0 \}.$ Note that every $\alpha \in \Z[G] $ can be written uniquely as 
\[ \alpha = \sum_{m \in \Supp(\alpha)} a_m [m]. \] 

\end{definition}

\begin{definition} \label{def defined}
   Let $d \geq 1$ and $F \in \Q(x).$
   \begin{enumerate}
       \item 
  We say that $F$ is defined at a complex number $\xi$ if $\xi$ is not a pole of $F.$ Equivalently, we can write $F=\frac{P}{Q}$ where $P, Q \in \Q[x]$ such that $Q(\xi) \neq 0.$ 
  In this case, we also define the value $F(\xi)$ as $
  \frac{P(\xi)}{Q(\xi)}$. It is easy to check that $F(\xi)$ is well-defined.
  
  \item We say that $\Phi_d$ is a factor of $F$ if  we can write $F=\frac{P}{Q}$ where $P, Q \in \Q[x]$ such that $\Phi_d\mid P(x)$ and $\Phi_d \nmid Q(x)$. Note that since $\Phi_d$ is irreducible, this is also equivalent to the fact that $F$ is defined at $\zeta_d$ and $F(\zeta_d)=0$. 
   \end{enumerate}
\end{definition}

\begin{lem} Let $f\in \Q(x)$ and $\xi\in \C$.
\begin{enumerate}
\item Suppose $f=ag+bh$, where $a,b\in \Z$ and $g,h\in \Q(x)$. If $g$ and $h$ are defined at $\xi$ then $f$ is also defined at $\xi$ and $f(\xi)=ag(\xi)+bh(\xi)$. 
\item Let $\alpha=\sum\limits_{a\in \supp(\alpha)} n_a[a]$. If $f$ is defined at $\xi^a$ for each $a\in\supp(\alpha)$ then $\alpha f$ is defined at $\xi$ and $\alpha f(\xi) =\sum\limits_{a\in \supp(\alpha)} n_a f(\xi^a)$.
\end{enumerate}

\end{lem}
\begin{proof} The proof is straightforward.
\end{proof}

With these above terminologies, we can now state our proposition. 

\begin{thm} \label{prop:0mod-rational}
    Let $\alpha\in  \Z[\Z^\circ]$, $d>1$ and $f(x)$ a rational function. Suppose that $\phi_d(\alpha)=0$ and $f(x)$ is defined at $\zeta_d^m$, for every $m \in \supp(\alpha)$. Then $\Phi_d$ is a factor of $\alpha f$.
\end{thm}
\begin{proof} Write $\alpha=\sum_{m\in \supp(\alpha) } a_m[m]$. Since     $\phi_d(\alpha)=0$, for each residue class $r\mod d$, we have 
\[\sum_{\substack{m\equiv r\mod d\\ m\in \supp(\alpha)} } a_m=0.\] 
Therefore
\[
\alpha f(\zeta_d) = \sum_{m\in \supp(\alpha) } a_m f(\zeta_d^m) =\sum_{r\mod d} \left(\sum_{\substack{m\equiv r\mod d\\ m\in \supp(\alpha)} } a_m\right) f(\zeta_d^r)=0.
\qedhere
\]
\end{proof}
We can translate the algebraic condition  $\phi_d(\psi_n)=0$ into explicit combinatorics of divisors as follows.
\begin{prop} \label{prop:paired-1}
   
Let \(n>1\) be a positive squarefree integer and let \(d>1\). The following are equivalent.

\begin{enumerate}
     \item The divisors of \(n\) can be paired so that
      \(\mu(a)+\mu(b)=0\) and \(a\equiv b\pmod d\) for every pair $\{a,b\}$.
    \item The divisors of some $N\mid n$, $N>1$, can be paired as in (1).
            \item \(\phi_d(\psi_n)=0\) in \(\mathbb Z[(\mathbb Z/d\mathbb Z)^\circ]\).
\end{enumerate}
\end{prop}
\begin{proof}
$(1)\Rightarrow (2)$: This is trivial.

$(2)\Rightarrow (3)$: Suppose that the set of divisors of $N\mid n$ can be paired as $\{a_1,b_1\},\ldots, \{a_l,b_l\}$ such that $\mu(a_i)+\mu(b_i)=0$ and  \(a\equiv b\pmod d\) for every $i$. We have
    \begin{align*}
        \psi_N
&= \sum_{m \mid N} \mu(m) [m] =  \sum_{i=1}^{l}
\mu(a_i)[a_i]+ \mu(b_i)[b_i]
\\
&=  \sum_{i=1}^{l}
\mu(a_i) ([a_i]-[b_i]).
    \end{align*}
    Since $d \mid a_i-b_i,$ for all $i \in \{1,2,\ldots,l\},$ we have $\phi_d(\psi_N)= 0.$ Because $n$ is squarefree, $\psi_n = \psi_N \psi_{n/N}$ and hence $\phi_d(\psi_n)=\phi_d(\psi_N)\phi_d(\psi_{n/N})=0$.

$(3)\Rightarrow (1)$:     For each residue class $r$, let
\[
S_r=\{a\mid n : a\equiv r\pmod d\},
\]
and decompose it according to the sign of $\mu$:
\[
S_r^{+}=\{a\in S_r : \mu(a)=1\}, 
\qquad
S_r^{-}=\{a\in S_r : \mu(a)=-1\}.
\]
Set $p_r = |S_r^{+}|$ and $q_r = |S_r^{-}|$.  

Since \( \phi_d(\psi_n)=0, \) we have \(  \sum_{a\mid n}\mu(a)[\overline{a}]=0 \) in \(\mathbb Z[(\mathbb Z/d\mathbb Z)^\circ]\). Equating the coefficient of each residue class \(r\) gives \(  p_r-q_r=0, \) that is, \( p_r=q_r \) for every residue class \(r\). 
Hence, for each residue class \(r\), the sets \(S_r^+\) and \(S_r^-\) have the same cardinality. We may therefore pair each element of \(S_r^+\) with a unique element of \(S_r^-\). 
Thus the set of divisors of \(n\) can be partitioned into pairs \[ \{a_1,b_1\},\ldots,\{a_\ell,b_\ell\} \] such that \( a_i\equiv b_i \pmod d \) and \( \mu(a_i)+\mu(b_i)=0 \) for every \(i\). 
\end{proof}

We illustrate \cref{prop:0mod-rational} with some examples.

\begin{expl}
We first provide a simple and straightforward example. Let $f=\dfrac{x}{x^2+2}.$ In this case $M_{15}(f)$ can be written in lowest terms as $\dfrac{P}{Q}$ with $P,Q\in\Q(x)$ and $\gcd(P,Q)=1.$ Moreover
\[
P=\Phi_1\Phi_2\Phi_4\,R(x),
\]
where $R(x)$ is an irreducible polynomial of degree $42.$ One checks that $\psi_{15}\equiv 0\mod [d]$ for $d\in\{1,2,4\}$. For example, for $d=4$ we have 
\[ \psi_{15} \equiv [1]-[3]-[5]+[15] \equiv [1]-[3]-[1]+[3] \equiv 0 \mod{[4]}.\]
{Equivalently, for $d=4$, the divisors of $n=15$ can be paired as $\{1, 5\}, \{3, 15\}$ as described in \cref{prop:paired-1}. }
Since $x^2+2$ has no roots of unity, $M_{15}(f):=\psi_{15}f$ is defined at every root of unity. Hence \cref{prop:0mod-rational} accounts for all cyclotomic factors of $M_{15}(f)$ in this case.
\end{expl}

\begin{expl} \label{exp:additional-root}
Consider a similar function $f=\dfrac{x}{1+x^2}$. Here the reduced form of $M_{15}(f)$ is $\dfrac{P}{Q}$ with
\[
P=\Phi_1\Phi_2\Phi_3\Phi_6\,R(x),
\]
where $R(x)$ is irreducible. Note that $\Phi_4$ is no longer a factor of $M_{15}(f)$. The reason \cref{prop:0mod-rational} does not apply is that $f$ is not defined at $\zeta_4$, so $M_{15}(f)$ is also undefined at $\zeta_4$. On the other hand $\Phi_6$ appears as a factor even though $\psi_{15}\not\equiv 0\mod{[6]}$. Indeed, modulo $[6]$ we have
\[
\psi_{15}\equiv[1]-[3]-[5]+[15]\equiv[1]-[3]-[5]+[3]\equiv[1]-[5]\equiv[1]-[-1] \not \equiv 0\mod{[6]}.
\]
This example shows we need additional tools to detect cyclotomic factors of $M_n(f)$. In the next section we develop a systematic method to partially achieve this.
\end{expl}

\subsection{Symmetries of $f$}
As explained in \cref{exp:additional-root}, while \cref{prop:0mod-rational} provides a rather general condition for the existence of certain cyclotomic factors of $M_n(f)$, there are some factors that cannot be explained by this proposition. Here is another example that we have known for a long time. Indeed, when $f_0=\dfrac{x}{1-x}$, we show in \cite{chidambaram2023fekete} that $\Phi_8$ is a factor of $M_{15}(f_0)$ even though $\psi_{15}$ is not congruent to $0$ mod $[8].$ It turns out that, in this case, $f_0$ has an extra symmetry, namely 
\[ f_0(x)+f_0 \left( \frac{1}{x} \right)=-1.\]
In this section, we provide a general theorem that explains this phenomenon. 

\iffalse
To do so, we need to enlarge the Mahler algebra to make sense of elements such as $\dfrac{1}{x}$. 

\tung{This definition is already mentioned in the previous section. We can remove it.}
\begin{definition}
We define the Mahler algebra $R:=\Z[\Z^{\circ}]$ to be the algebra generated by the monoid $\Z^{\circ}:=\Z \setminus \{0\}.$
\end{definition}

\fi

%Let $M= \Q(x)$. Then $M$ is module over $\Z[\Z^{\circ}]$ under the natural action; namely if $\alpha = \sum_{k} n_k [\alpha_k]$ then 
%\[ \alpha(f) = \sum_{k} n_k f(x^{\alpha_k}). \]

For an element $f \in M:= \Q(x)$, the symmetries of $f$ can be described by an annihilator ideal--which we now recall. 
\begin{definition}
Let $f \in M$ be an element in $M$. We define the annihilator ideal of $f$ as 
\[ \Ann_{R}(f)=\{r \in R \mid rf = 0\}.\]
We refer to elements of \(\Ann_R(f)\) as symmetries of \(f\).
\end{definition}

\begin{expl}
    If $f(x)=\dfrac{x}{x^2+1}$ as in \cref{exp:additional-root}, then we can see that $[1]-[-1] \in \Ann_{R}(f)$. More precisely 
    \[ ([1]-[-1])f(x)= \dfrac{x}{1+x^2}-\dfrac{1/x}{1+1/x^2} = \dfrac{x}{1+x^2}-\dfrac{x}{1+x^2}=0.\]

\end{expl}
We consider  another example which is somewhat less trivial. 
\begin{expl}\label{ex:modifiedFekete}
    
    While it is not true that $([-1]+[1])f_0 =0$ where $f_0 = \dfrac{x}{1-x}$, we can check that  
    \[ ([-1]+[1])f_0 = \frac{x}{1-x}+\frac{1/x}{1-1/x} = -1.\]
Therefore, if we let $f_1= f_0 +\dfrac{1}{2}$ then $([-1]+[1])f_1=0$, and therefore $[-1]+[1] \in \Ann_{R}(f_1).$ However, observe that for $n>1$ 
\[\psi_n(f_1)=\psi_n(f_0)+\psi_n(1/2)=\psi_n(f_0).\]
As a result, we can use $f_1$ to study the cyclotomic factors of $M_n(f_0).$ More precisely, for \(n>1\), the cyclotomic factors of \(M_n(f_0)\) may be studied using the symmetry given by: \(([1]+[-1])f_1=0\).

\end{expl}

We now state a general criterion relating symmetries of \(f\) to cyclotomic factors of \(\alpha f\). % We remark that when $I=0$;  this result covers \cref{prop:0mod-rational}.

\begin{thm} \label{thm:mahler-congruence_revised}
   % Let $I\subset R=\Z[\Z^{\circ}]$ be an ideal and  $f\in \Q(x)$ with $I\subset  \Ann_{R}(f)$. Let $\alpha \in \Z[\Z^{\circ}]$.
   Let $f\in \Q(x)$ and $\alpha\in R=\Z[\Z^{\circ}]$.
    Suppose there exists $\beta\in \Ann_R(f)$ such that $\phi_d(\alpha)=\phi_d(\beta)$ and $f$ is defined at $\zeta_d^a$ for each $a \in \supp(\alpha)\cup \supp(\beta)$. Then $\Phi_d$ is a factor of $\alpha f$. 
\end{thm}
\begin{proof}
    Since $\beta\in \Ann_R(f)$, $\beta f=0$ and hence $\alpha f=(\alpha-\beta )f$ in $\Q(x)$. 
    Note that $\supp(\alpha-\beta)\subset \supp(\alpha)\cup \supp(\beta)$. Hence $f$ is defined at $\zeta_d^a$ for each $a\in \supp(\alpha-\beta)$. By Proposition~\ref{prop:0mod-rational}, $\Phi_d$ is a factor of $(\alpha-\beta )f=\alpha f$.
\end{proof}

{We now describe some equivalent combinatorial conditions for some special cases of the congruences described in \cref{thm:mahler-congruence_revised}. We first deal with the case where the symmetry is governed by $[1]+[-1]$ as in the case of Fekete polynomials.}

\begin{lem} \label{lem:[-1]+[1]-first}
  %  Let $I = \langle [1] + [-1] \rangle$ be the ideal generated by $[1]+[-1]$ in the ring $\Z[\Z^{\circ}]$. 
    Let $\alpha=\sum\limits_{a\in \supp(\alpha)} n_a[a]\in \Z[\Z^{\circ}]$. For each residue class $r\in (\Z/d)^\circ$, define $c_r=\sum\limits_{\substack{a\equiv r\mod d\\ a\in \supp(\alpha)}}n_a$. 
    Then $\phi_d(\alpha)\in \phi_d(\langle [1] + [-1] \rangle)$ if and only if $c_r=c_{-r}$ for $r\not\equiv -r\pmod d$ and $c_r$ is even for $r\equiv -r\pmod d$.
   \end{lem}
   \begin{proof}
       Clearly, we have $\phi_d(\alpha)=\sum c_r[r]$. On the other hand, $\phi_d(\langle [1] + [-1] \rangle)= \langle [1] + [-1] \rangle \subset \Z[(\Z/d)^\circ]$ and an arbitrary element in $\phi_d(\langle [1] + [-1] \rangle)$ is of the form
       \[
       \left(\sum_{r} m_r [r]\right) ([1] + [-1]) = \sum_{r} m_r ([r] + [-r]).
       \]
       From this we see that $\phi_d(\alpha)\in \phi_d(\langle [1] + [-1] \rangle)$ if and only if $c_r=c_{-r}$ for $r\not\equiv -r\pmod d$, and $c_r$ is even for $r\equiv -r\pmod d$. Note that $r\equiv -r\pmod d$ if and only if $2r\equiv 0\pmod d$ if and only if $r=0$ or $r=d/2$ if $d$ is even.
   \end{proof}

\begin{lem} \label{lem:[-1]+[1]-second}
   % Let $I = \langle [1] + [-1] \rangle$ be the ideal generated by $[1]+[-1]$ in the ring $\Z[\Z^{\circ}]$. 
    Let $\alpha\in \Z[\Z^{\circ}]$. Suppose that $\phi_d(\alpha)\in \phi_d(\langle [1] + [-1] \rangle)$. Then there exists $\beta\in \langle [1] + [-1] \rangle$ such that $\phi_d(\alpha)=\phi_d(\beta)$ and $\supp(\beta)\subseteq \supp(\alpha)\cup(-\supp(\alpha))$. Here, $-\supp(\alpha)=\{-a\mid  a\in \supp(\alpha)\}.$
     \end{lem}

     \begin{proof} Write $\alpha=\sum\limits_{a\in \supp(\alpha)} n_a[a]$.  For each residue class $r\in (\Z/d)^\circ$, define $c_r=\sum\limits_{\substack{a\equiv r\mod d\\ a\in \supp(\alpha)}}n_a$. By the previous lemma, $c_r=c_{-r}$ for $r\not\equiv -r\pmod d$ and $c_r$ is even for $r\equiv -r\pmod d$. 
     Now for each pair of non-self-inverse classes $\{r,-r\}$ such that $c_r\not=0$, we choose a representative $a_r\in \supp(\alpha)$ and set $\beta_{\{r,-r\}}= c_r([a_r]+[-a_r])$. For class $r$ with $r\equiv-r\pmod d$, we have $c_r=2 c^\prime_ r$, and if $c_r\not=0$ then we choose a representative $a_r\in \supp(\alpha)$ and set $\beta^\prime_{r}= c^\prime_r ([a_r]+[-a_r])$. Let $\beta$ be the sum of those $\beta_{\{r,-r\}}$ and $\beta^\prime_r$. Then clearly $\beta$ is in $\langle [1] + [-1] \rangle$, and $\supp(\beta)\subseteq \supp(\alpha)\cup(-\supp(\alpha))$ and $\phi_d(\beta)=\phi_d(\alpha)$. 
     \end{proof}
     \begin{expl}
         Let $I = \langle [1] + [-1] \rangle \subset \Z[\Z^{\circ}]$. Let $d=6$  and $\alpha=3[7]+[-5] + 2[5]+2[11] +4[3]$. In $\Z[(\Z/d)^\circ]$, we have
         \[\phi_d(\alpha)= 4[1]+4[-1] +2[3]+2[-3]= (4[1]+2[3]) ([1]+[-1])\in \phi_d(I).
         \]
         Here we can choose $\beta=4([7]+[-7])+2([3]+[-3])\in I$. Clearly $\beta\in I$, $\supp(\beta)\subseteq \supp(\alpha)\cup(-\supp(\alpha))$ and $\phi_d(\beta)=\phi_d(\alpha)$.
     \end{expl}

Combining \cref{lem:[-1]+[1]-first} and \cref{lem:[-1]+[1]-second}, we have the following theorem.  
\begin{thm} \label{thm:mahler-congruence}
   % Let $I$ be the ideal in $\Z[\Z^{\circ}]$ generated by $[1] + [-1]$. %Let $\bar{I}$ be the image of $I$ under the canonical map $\phi_d\colon \Z[\Z^{\circ}] \to \Z[(\Z/d)^{\circ}].$ 
    Suppose $\alpha \in \Z[\Z^{\circ}]$ such that $\phi_d(\alpha) \in \phi_d(\langle [1] + [-1] \rangle).$ Let $f \in \Q(x)$ such that $[1] + [-1] \in \Ann_{\Z[\Z^{\circ}]}(f)$. If $f$ is defined at $\zeta_d^a$ for each $a \in \supp(\alpha)$, then $\Phi_d$ is a factor of $\alpha f$. 
\end{thm}
\begin{proof}
    Since $I:=\langle [1] + [-1] \rangle \subset \Ann_{\Z[\Z^{\circ}]}(f)$, one has $([1]+[-1])f=0$, i.e., $f(x)+f(1/x)=0$. From this we see that if $f$ is defined at $\zeta_d^a$  then $f$ is also defined at $\zeta_d^{-a}$. Hence $f$ is defined at $\zeta_d^a$ for each $a\in \supp(\alpha)\cup(-\supp(\alpha))$.

    By the previous lemma, there exists $\beta\in I$ such that $\phi_d(\beta)=\phi_d(\alpha)$ and
    \(
    \supp(\beta)\subseteq \supp(\alpha)\cup(-\supp(\alpha)).
    \)
    Clearly, \( \supp(\alpha)\cup\supp(\beta) \subseteq \supp(\alpha)\cup(-\supp(\alpha)).\) Hence $f$ is defined at $\zeta_d^a$ for each $a\in \supp(\alpha)\cup \supp(\beta)$. 
    By Theorem~\ref{thm:mahler-congruence_revised}, $\Phi_d$ is a factor of $\alpha f$. 
\end{proof}

\begin{rem}
Theorem~\ref{thm:mahler-congruence} generalizes the main theorem of~\cite[Theorem~2.19]{chidambaram2023fekete}
concerning cyclotomic factors of Fekete polynomials. Indeed, let
\[
f(x)=\frac{x}{1-x}+\dfrac{1}{2}.
\]
As observed earlier in Example \ref{ex:modifiedFekete},
\(
([1]+[-1])f=0
\)
and for \(n>1\),
\(
\psi_n(f)=\tilde{F}_n(x).
\)
Therefore, applying Theorem~\ref{thm:mahler-congruence}
to \(f\) and $\alpha=\psi_n$ recovers the cyclotomic factors of the modified Fekete
polynomials obtained in~\cite[Theorem~2.19]{chidambaram2023fekete}.
    
\end{rem}
   By a similar argument as above, we obtain the following result.
   
     \begin{thm} \label{thm:mahler-congruence_02}
   % Let $I$ be an ideal in $\Z[\Z^{\circ}]$ generated by $[1] - [-1]$. %Let $\bar{I}$ be the image of $I$ under the canonical map $\phi_d\colon \Z[\Z^{\circ}] \to \Z[(\Z/d)^{\circ}].$ 
    Suppose $\alpha \in \Z[\Z^{\circ}]$ such that $\phi_d(\alpha) \in \phi_d(\langle [1] - [-1]\rangle).$ Let $f \in \Q(x)$ such that $[1] - [-1]\in \Ann_{\Z[\Z^{\circ}]}(f)$. If $f$ is defined at $\zeta_d^a$ for each $a \in \supp(\alpha)$, then $\Phi_d$ is a factor of $\alpha f$. 
\end{thm}

\begin{expl} We consider the function $f(x)=\dfrac{x}{1+x^2}$ as in Example~\ref{exp:additional-root}. 
    We now explain why the factor $\Phi_6$ of $M_{15}(f)$ can be explained via  \cref{thm:mahler-congruence_02}.  We have \[ f(1/x) = \frac{1/x}{1+1/x^2} = \frac{x}{1+x^2} = f(x). \] 
    Hence \( ([1]-[-1])f=0, \) and therefore \( \langle [1]-[-1]\rangle \subseteq \operatorname{Ann}_{\mathbb Z[\mathbb Z^\circ]}(f). \)
    
    On the other hand, \[ \psi_{15}=[1]-[3]-[5]+[15]. \] Reducing modulo $6$, we obtain \[ \phi_6(\psi_{15}) = [\overline{1}]-[\overline{3}]-[\overline{5}]+[\overline{15}] =   [\overline{1}]-[\overline{-1}]  \in \phi_6\bigl(\langle [1]-[-1]\rangle\bigr). \] 
    Moreover,  $f$ is defined at $\zeta_6^a$ for every $a\in \operatorname{Supp}(\psi_{15})$. Therefore, by Theorem \ref{thm:mahler-congruence_02}, $\Phi_6$ is a factor of \(  M_{15}(f)=\psi_{15}f. \)
\end{expl}
The next proposition translates the algebraic condition  $\phi_d(\psi_n)\in \phi_d(\langle [1] + [-1] \rangle)$ into explicit combinatorics of divisors.
\begin{prop} \label{pro:equivalence}
   Let $n>1$ be a positive squarefree integer and  $d>1$ be a positive integer. %and $d \neq n$. 
   The following conditions are equivalent.
   \begin{enumerate}
       \item The divisors of \(n\) can be paired so that
      \(\mu(a)a+\mu(b)b\equiv 0 \pmod d\)
      for every pair \(\{a,b\}\).
          \item The divisors of some $N\mid n$, $N>1$, can be paired as in (1).
            \item $\phi_d(\psi_n)\in \phi_d(\langle [1] + [-1] \rangle)$.% where $I$ is the principal ideal generated by $[1]+[-1].$
   \end{enumerate}
   \end{prop}
\begin{proof}
    $(1)\Rightarrow (2):$ Take $N := n.$ The conclusion follows immediately.
    
    $(2)\Rightarrow (3):$ Suppose that the set of divisors of $N$ can be partitioned  into pairs 
        \[
\{a_1,b_1\},\{a_2,b_2\},\ldots,\{a_\ell,b_\ell\}, \] 
such that $d\mid \mu(a_i)a_i+\mu(b_i)b_i,$ for every $i$. 
    Let
\[
\alpha_1=\sum_{\mu(a_i)=\mu(b_i)}(\mu(a_i)[a_i]+\mu(b_i)[b_i]),\quad\text{ and } \alpha_2=\sum\limits_{\mu(a_i)=-\mu(b_i)}(\mu(a_i)[a_i]+\mu(b_i)[b_i]).
\]
    We have $\psi_N=\sum_{a\mid N}\mu(a)[a]=\alpha_1+\alpha_2$.  
    
For each pair $(a_i,b_i),$ we have \(d\mid \mu(a_i)a_i+\mu(b_i)b_i.\) 
If $\mu(a_i)=\mu(b_i),$
then $d\mid(a_i+b_i)$ and $b_i\equiv -a_i\pmod d$. Hence 
\[
\phi_d(\alpha_1)
 =\sum_i\mu(a_i)\bigl([\overline{a_i}]+[\overline{-a_i}]\bigr)
 \in\phi_d(I).
\]
If $\mu(a_i)=-\mu(b_i),$
then $d\mid(a_i-b_i)$ and
$
\phi_d(\alpha_2)=\sum\limits_{\mu(a_i)=-\mu(b_i)}(\mu(a_i)[\overline{a_i}]+\mu(b_i)[\overline{b_i}]) =0.
$
Thus
\[
\phi_d(\psi_N)=\phi_d(\alpha_1+\alpha_2)=\phi_d(\alpha_1)\in\phi_d(I).
\]
Since $\phi_d$ is a surjective ring homomorphism, $\phi_d(I)$ is an ideal of $\Z[(\Z/d)^\circ]$. 
Since $n$ is squarefree  and $N\mid n$, $\gcd(N,n/N)=1$ and  $\psi_n = \psi_N \psi_{n/N}.$ Hence
\[
\phi_d(\psi_n)=\phi_d(\psi_N)\phi_d(\psi_{n/N})\in \phi_d(I).
\]

$(3)\Rightarrow (1):$
For each residue class $r$, let
\[
S_r=\{a\mid n : a\equiv r\pmod d\},
\]
and decompose it according to the sign of $\mu$:
\[
S_r^{+}=\{a\in S_r : \mu(a)=1\}, 
\qquad
S_r^{-}=\{a\in S_r : \mu(a)=-1\}.
\]
Set $p_r = |S_r^{+}|$ and $q_r = |S_r^{-}|$.  

The condition for $\phi_d(\psi_n)\in \phi_d(I)$ is equivalent to 
\[
p_r - q_r = p_{-r} - q_{-r},
\]
 for all non-self-inverse residue classes $r\not=-r$, and
\[
p_r - q_r\equiv 0\pmod 2,
\]
for self-inverse residue classes $r=-r$.

Consider a non-self-inverse class $r$ ($r\not=-r$).
Without loss of generality, we may assume that $p_r \geq p_{-r}$. 
Then $s:=q_r-q_{-r}=p_{r}-p_{-r}\geq0$. We pair $p_{-r}$ elements in $S_r^+$ with $p_{-r}$ elements in $S_{-r}^+$.
We pair $q_{-r}$ elements in $S_r^-$ with $q_{-r}$ elements in $S_{-r}^-$. Clearly, for such a pair $\{a,b\}$ we have $a\equiv -b\pmod d$ and $\mu(a)=\mu(b)$. 
The remaining $s$ elements in $S_r^{+}$ are paired with $s$ elements in $S_{r}^{-}$. For such a latter pair $\{a,b\}$ we have $a\equiv b\equiv r\pmod d$ and $\mu(a)=-\mu(b)$. In any case, for each  obtained pair $\{a_i,b_i\}$, 
\[
d \mid \mu(a_i)a_i+\mu(b_i)b_i.
\]

Now consider a self-inverse class $r$ ($r=-r$). We have $|S_r|=p_r+q_r\equiv p_r-q_r\equiv 0\pmod 2$. Hence $S_r$ has an even number of elements. We pair the elements of $S_r$ arbitrarily.  For such a pair $\{a,b\}$ we have $a\equiv b\equiv r\equiv -b\pmod d$. Hence $d\mid \mu(a)a+\mu(b)b$.
\end{proof}
\begin{cor}
Suppose that $\phi_d(\psi_n)\in \phi_d(\langle[1]+[-1]\rangle)$. Then  $d \mid \varphi(n)= \prod_{p \mid n}(p-1).$
\end{cor}
\begin{proof}
We recall the following identity \cite[Pg.~81, (4.5)]{hyde2019polynomial}. For any positive integer $n$, it is well known that
\[
\varphi(n) = \sum_{e \mid n} \mu(e)\,\frac{n}{e}.
\]
In particular, when $n$ is squarefree, this identity reduces to
\[
\varphi(n) = \mu(n) \sum_{e \mid n} \mu(e)\, e.
\]

    By Proposition~\ref{pro:equivalence}, there exists a partition of  the set of all divisors of $n$ into pairs 
        \[
\{a_1,b_1\},\{a_2,b_2\},\ldots,\{a_\ell,b_\ell\}, \] 
such that $d\mid \mu(a_i)a_i+\mu(b_i)b_i$ for every $i$. Summing over all pairs gives
\[d\mid \sum_{i=1}^\ell  \mu(a_i)a_i+\mu(b_i)b_i = \sum_{e\mid n} \mu(e)e=\mu(n)\varphi(n).\]
 Hence $d\mid\varphi(n)$.
\end{proof}
{When $d$ is a prime number, we have the following simple corollary. 
\begin{cor} \label{cor:d-prime}
    Let $d$ be a prime number. Then $\phi_d(\psi_n)\in \phi_d(\langle [1] + [-1] \rangle)$ if and only if there exists a prime divisor $p$ of $n$ such that $d \mid p-1.$
\end{cor}

\begin{proof}
    Suppose that $\phi_d(\psi_n)\in \phi_d(\langle [1] + [-1] \rangle)$. Then $d \mid \varphi(n)=\prod_{p \mid n}(p-1).$ Since $d$ is a prime number, there exists $p$ such that $d \mid p-1.$ Conversely, suppose that $d \mid p-1$ for some prime divisor $p$ of $n.$ Then, we can apply the second criterion in \cref{pro:equivalence} for $N=p$. 
\end{proof}

}

We demonstrate \cref{pro:equivalence} by some examples. 
\begin{expl}
We first consider the case $n=15$ and $d=8.$ In this case, a partition of $\Div(15)$ that works is $\{1, 15\}, \{3, 5\}.$ In fact, we can check directly that $\phi_{8}(\psi_{15}) \in \phi_8(\langle [1]+[-1] \rangle )$. 
    \[ \psi_{15}=[1]-[3]-[5]+[15]=[1]-[3]-[-3]+[-1]=([1]+[-1])([1]-[3]) \mod{[8]}.\]
Similarly, we can check that for $n=5 \times 7$ and $d=12$, a partition that works is $\{1, 35\}, \{5, 7\}.$
\end{expl}

\begin{expl}
\label{ex:3x5x7}
We  consider $n$ with more prime factors, for example $n=3 \times 5 \times 7=105.$ Let $d=16$. The following partition works for $d=16$.
\[ \{1, 15 \}, \{3, 35 \}, \{5, 21\}, \{7, 105\}. \]
We can also check that for each $N \mid 105$ and $N\not=105$, $\phi_{12}(\psi_N)$ does not belong to $\phi_{12}(\langle [1]+[-1] \rangle ).$

Let $f=\dfrac{1-x}{x+1}$. Then, we can see that $([1]+[-1])f=0.$ \cref{thm:mahler-congruence} and \cref{pro:equivalence} would imply that $\Phi_{16}$ is a factor of $M_{105}(f).$ In fact, using Sagemath, we can check that if we write $M_{105}(f)=\dfrac{P}{Q}$ in the reduced form, then 
\[ P=\Phi_1 \Phi_3 \Phi_4 \Phi_8 \Phi_{16} \Phi_{24} R(x), \]
where $R(x) \in \Q[x].$ We can verify that, except for $d=24$, all cyclotomic factors of $M_{105}(f)$ can be explained by \cref{thm:mahler-congruence} and \cref{pro:equivalence}. The case $d=24$ will be explained in \cref{sec:galois} using explicit calculations. 
\end{expl}

We now study a similar question as \cref{pro:equivalence} for the ideal generated by $[1]-[-1].$

\begin{prop} \label{pro:equivalence_02}
   Let $n>1$ be a positive squarefree integer and  $d>1$ be a positive integer. %and $d \neq n$. 
   The following conditions are equivalent.
   \begin{enumerate}
       \item The divisors of \(n\) can be paired so that
      \(\mu(a)+\mu(b)=0\) and either \(a\equiv b \pmod d\) or $a\equiv -b\pmod d$
      for every pair \(\{a,b\}\). 
          \item  The divisors of some $N\mid n$, $N>1$, can be paired as in (1).
           \item $\phi_d(\psi_n)\in \phi_d(\langle [1] - [-1] \rangle)$.% where $I$ is the principal ideal generated by $[1]-[-1].$
   \end{enumerate}
   \end{prop}
\begin{proof}
    $(1)\Rightarrow (2):$ Take $N := n.$ The conclusion follows immediately.
    
    $(2)\Rightarrow (3):$ Suppose that the set of divisors of $N$ can be partitioned  into pairs 
        \[
\{a_1,b_1\},\{a_2,b_2\},\ldots,\{a_\ell,b_\ell\}, \] 
such that $\mu(a_i)+\mu(b_i)=0$ and either $d\mid a_i-b_i$ or $d\mid a_i+b_i$, for every $i$.  Let $I=\langle[1]-[-1]\rangle$.
    Let
\[
\alpha_1=\sum_{d\mid a_i+b_i}(\mu(a_i)[a_i]+\mu(b_i)[b_i]),\quad\text{ and  } \quad \alpha_2=\sum\limits_{d\mid a_i-b_i}(\mu(a_i)[a_i]+\mu(b_i)[b_i]).
\]
   Clearly, $\psi_N=\sum_{a\mid N}\mu(a)[a]=\alpha_1+\alpha_2$.  
    
We have 
\[
\phi_d(\alpha_1)
 =\sum_{d\mid (a_i+b_i)}\mu(a_i)\bigl([\overline{a_i}]-[\overline{-a_i}]\bigr)
 \in\phi_d(I),
 \]
and 
$
\phi_d(\alpha_2)=\sum\limits_{d\mid a_i-b_i}(\mu(a_i)[\overline{a_i}]-\mu(b_i)[\overline{b_i}]) =0.
$
Thus
\[
\phi_d(\psi_N)=\phi_d(\alpha_1+\alpha_2)=\phi_d(\alpha_1)\in\phi_d(I).
\]
Since $\phi_d$ is a surjective ring homomorphism, $\phi_d(I)$ is an ideal of $\Z[(\Z/d)^\circ]$. 
Since $n$ is squarefree  and $N\mid n$, $\gcd(N,n/N)=1$ and  $\psi_n = \psi_N \psi_{n/N}.$ Hence
\[
\phi_d(\psi_n)=\phi_d(\psi_N)\phi_d(\psi_{n/N})\in \phi_d(I).
\]

$(3)\Rightarrow (1):$
For each residue class $r$, let
\[
S_r=\{a\mid n : a\equiv r\pmod d\},
\]
and decompose it according to the sign of $\mu$:
\[
S_r^{+}=\{a\in S_r : \mu(a)=1\}, 
\qquad
S_r^{-}=\{a\in S_r : \mu(a)=-1\}.
\]
Set $p_r = |S_r^{+}|$ and $q_r = |S_r^{-}|$.  

The condition for $\phi_d(\psi_n)\in \phi_d(I)$ is equivalent to 
\[
p_r - q_r =-(p_{-r} - q_{-r}),
\]
all non-self-inverse residue classes $r\not=-r$, and
\[
p_r - q_r= 0,
\]
for self-inverse residue classes $r=-r$.

Consider a non-self-inverse class $r$ ($r\not=-r$).
Without loss of generality, we may assume that $p_r \geq q_{-r}$. 
Then $s:=p_r-q_{-r}=q_{r}-p_{-r}\geq0$. We pair $q_{-r}$ elements in $S_r^+$ with $q_{-r}$ elements in $S_{-r}^-$.
We pair $p_{-r}$ elements in $S_{-r}^+$ with $p_{-r}$ elements in $S_{r}^-$. Clearly, for such a pair $\{a,b\}$ we have $a\equiv -b\pmod d$ and $\mu(a)=-\mu(b)$. 
The remaining $s$ elements in $S_r^{+}$ are paired with the remaining $s$ elements in $S_{r}^{-}$. For such a latter pair $\{a,b\}$ we have $a\equiv b\equiv r\pmod d$ and $\mu(a)=-\mu(b)$. 

Now consider a self-inverse class $r$ ($r=-r$). In this case $|S_r^+|=|S_r^-|$. We pair each element in $S_r^+$ with an element in $S_r^-$.  For such a pair $\{a,b\}$ we have $\mu(a)=-\mu(b)$ and $a\equiv b\equiv r\pmod d$. 
\end{proof}

\begin{cor}
If \(\phi_d(\psi_n)\in \phi_d(\langle [1]-[-1]\rangle)\), then
\[
d\mid J_2(n)=\prod_{p\mid n}(p^2-1).
\]
\end{cor}

\begin{proof}
By \cref{pro:equivalence_02}, the divisors of \(n\) can be paired so that
\[
\mu(a)+\mu(b)=0
\qquad\text{and}\qquad
a\equiv \pm b \pmod d.
\]
Hence
\[
\mu(a)a^2+\mu(b)b^2
=\mu(a)(a^2-b^2)
\equiv 0 \pmod d
\]
for every pair. Summing over all pairs yields
\[
\sum_{e\mid n}\mu(e)e^2\equiv 0 \pmod d.
\]
Since \(n\) is squarefree,
\[
\sum_{e\mid n}\mu(e)e^2
=
\prod_{p\mid n}(1-p^2)
=
(-1)^{\omega(n)}J_2(n).
\]
Therefore \(d\mid J_2(n)\).
\end{proof}

{
We have a similar corollary as \cref{cor:d-prime} with an identical proof. 

\begin{cor}
    Suppose that $d$ is a prime number. Then $\phi_d(\psi_n)\in \phi_d(\langle [1] - [-1] \rangle)$ if and only if there exists a prime divisor $p$ of $n$ such that either $d \mid p-1$ or $d \mid p+1.$
\end{cor}

}

\section{Signed Mahler algebra, symmetries, and cyclotomic factors of rational necklace functions} \label{sec:signed-mahler}

In this section, we introduce a variant, the signed Mahler algebra, of the Mahler algebra described in \cref{sec:mahler}. This variant is based on the observation that certain symmetries of $f$ can only be described by the signed Mahler algebra. For example, let us consider $f_2 = \dfrac{x}{1+x^2}$. Then 
\[ f_2(x)+f_2 \left(-\frac{1}{x} \right)=0.\]
The presence of the negative sign shows that we cannot use $\Z[\Z^{\circ}]$ to describe this symmetry. We can overcome this by introducing a sign component to the Mahler algebra $\Z[\Z^{\circ}]$ with one small modification: we can only take odd exponents (a similar signed Mahler algebra in \cite{hyde2019polynomial} has to deal with the same exact issue). More precisely, we have the following definition.

\begin{definition}
Let $\Z^{\odd}$ be the monoid of all odd integers under multiplication. Let $\{\pm{1} \}$ be the group of order $2$ generated by $-1.$  We define the signed Mahler algebra to be the algebra $R^{\pm}:=\Z[\{\pm 1\} \times \Z^{\odd}].$ 
\end{definition}
The signed Mahler algebra acts on $\Q(x)$ by the following rule: for a basis element $\alpha = (\epsilon, m)$
where $\epsilon \in \{\pm{1}\}$ 
\[ \alpha(f)=f(\epsilon x^m) = f((\epsilon x)^m). \]
The second equality follows from the fact that $m$ is odd (this is the reason we restrict ourselves to $\Z^{\odd}$). We  note that there is a natural embedding of $\Z[\Z^{\odd}]$ into $\Z[\{\pm 1\} \times \Z^{\odd}]$ defined by sending $[e] \mapsto [(1,e)].$ In particular, if $n$ is odd, then we can consider $\psi_n$ as an element of  $\Z[\{\pm 1\} \times \Z^{\odd}]$. 

\begin{expl}
    Let $\alpha=[(1, 1)]+[(-1, 1)]$. Then $\alpha f (x)=f(x)+f(-x)$. Therefore, $\alpha f = 0$ if and only if $f$ is an odd rational function. This is precisely the case studied in \cite[Section 4.2.1]{hyde2019polynomial}. 
\end{expl}

Here is another example where the symmetry of $f$ is somewhat more interesting. 
\begin{expl}
    Let $\alpha = [(-1, 1)]$. Then, $\alpha x^n = -x^n$ ($n$ is odd). On the other hand, if $\alpha = [(-1, -1)]$ then $\alpha x^n = -x^{-n}= -\dfrac{1}{x^n}.$ We see that this signed Mahler algebra can describe functional equations such that $f(x)+f(-1/x)=0.$
\end{expl}

The following proposition generalizes the natural projection map $\phi_d\colon \Z[\Z^{\circ}] \to \Z[(\Z/d)^{\circ}].$ To do so, we make the assumption that $d$ is a multiple of $4$ and $d_0$ is an even integer such that $d=2d_0.$ We will make this assumption throughout this section. 
\begin{prop}
\label{prop:0mod-rational_signed}
Let $\theta_d\colon \{\pm 1 \} \times \Z^{\odd} \to (\Z/d)^{\circ} $ be the map defined by
 \[ \theta_d((-1, m))= \overline{d_0+m}, \quad \theta_d((1,m))= \overline{m}, \forall m \in \Z^{\odd}. \]
 Then $\theta_d$ is a monoid homomorphism. 
\end{prop}

\begin{proof}
    Let $m,n \in \Z^{\odd}.$ We will show that 
    \[ \theta_d((-1,m)) \theta_d((-1, n))= \theta_d((1, mn)),\]
    and 
    \[ \theta_d((-1,m)(1, n))= \theta_d((-1, mn)) .\] 
For the first equality, we have 
\[ \theta_d((-1,m)) \theta_d((-1,n))=\overline{(d_0+m)(d_0+n)} = \overline{d_0^2 +d_0(m+n) + mn}.\]
Since $d_0$ is even, $d_0^2 \equiv 0 \pmod{d}.$ Additionally, since $m,n$ are both odd, $d_0(m+n) \equiv 0 \pmod{d}.$ We conclude that
\[ \theta_d((-1,m)) \theta_d((-1,n)) = \overline{mn} = \theta_d((1,mn)). \] 
%The second equality follows from the same argument. 
For the second equality, we have 
\[ \theta_d((-1,m)) \theta_d((1,n))= \overline{(d_0+m)n} = \overline{d_0n +mn}.\]
Since $n$ is odd, $d_0n\equiv d_0\pmod d$. Hence $d_0n+mn\equiv d_0+mn\pmod d$ and we conclude that
\[ \theta_d((-1,m)) \theta_d((1,n)) = \overline{d_0+mn} = \theta_d((-1,mn)).
\qedhere\] 
\end{proof}
Let $\theta_d\colon \Z[\{\pm 1 \} \times \Z^{\odd}] \to \Z[(\Z/d)^{\circ}]$ be the induced map on the ring level. 

The definition of the support of an element in $\Z[\{\pm 1 \} \times \Z^{\odd}]$ is defined naturally. For convenience, for $(\epsilon,m)\in \{\pm 1 \} \times \Z^{\odd}$ and $\xi\in \C$, we define $\xi^{(\epsilon,m)}=\epsilon \xi^m$.  
\begin{prop} \label{prop:0mod-rational_signed}
    Let $\alpha\in  \Z[\{\pm 1 \} \times \Z^{\odd}]$, $d=2d_0$ with $d_0$ even and $f(x)\in \Q(x)$ a rational function. Suppose that $\theta_d(\alpha)=0$ and $f(x)$ is defined at $\zeta_d^{(\epsilon,m)}$, for every $(\epsilon,m) \in \supp(\alpha)$. Then $\Phi_d$ is a factor of $\alpha f$.
\end{prop}
\begin{proof} Let $\zeta=\zeta_d$. We may write
\[
\alpha
= \sum_{(1,m)\in S_1} c_m[(1, m)]  +\sum_{(-1,n)\in S_2} d_n [(-1, n)],
\]
where $c_m, d_n \in \mathbb{Z},\; m, n \in \mathbb{Z}^{\mathrm{odd}}$, and $\supp(\alpha)=S_1\sqcup S_2$. 
  
Define
\[
\gamma
:= \sum_{(1,m)\in S_1} c_m [m]
+ \sum_{(-1,n)\in S_2} d_n \left[n + \frac{d}{2}\right]
\in \mathbb{Z}[\mathbb{Z}^{\circ}].
\]
Since \(d\) is even, we have \(\zeta^{\,d/2} = -1\). 
Note also that $f$ is defined at $\zeta^a$, for each $a\in \supp(\gamma)$.  If $a=m\in \supp(\gamma)$, then definedness follows from the hypothesis on $(1,m)\in S_1$. 
If $a=n+d/2\in \supp(\gamma)$, then $\zeta^a =-\zeta^n$, and the definedness follows from the hypothesis on $(-1,n)\in S_2$.

We have
\[
\begin{aligned}
\alpha f(\zeta)
&= \left( \sum_{(1,m)\in S_1} c_m[(1, m)]  +\sum_{(-1,n)\in S_2} d_n [(-1, n)] \right) f(\zeta) \\
&= \sum_{(1,m)\in S_1} c_m f(\zeta^m) +\sum_{(-1,n)\in S_2} d_n f(-\zeta^n)\\
&= \sum_{(1,m)\in S_1} c_m f(\zeta^m)  +\sum_{(-1,n)\in S_2} d_n f(\zeta^{n+d/2}) \\
&= \left( \sum_{(1,m)\in S_1} c_m [m]
+ \sum_{(-1,n)\in S_2} d_n \left[n + \frac{d}{2}\right] \right) f(\zeta) \\
&= \gamma f(\zeta).
\end{aligned}
\]
 
Moreover, in $\Z[(\Z/d)^\circ]$ we have  
\[
0=\theta_d (\alpha) =\sum_{(1,m)\in S_1} c_m [m] +  \sum_{(-1,n)\in S_2} d_n [n+d/2]=\phi_d(\gamma),
\]
where \(\phi_d\) is the natural homomorphism $\phi_d\colon \Z[\Z^\circ]\to \Z[(\Z/d)^\circ]$. 
By Proposition~\ref{prop:0mod-rational}, it follows that \(\gamma f(\zeta) = 0\).
\end{proof}

We have the following theorem which is a direct analog of \cref{thm:mahler-congruence_revised}.

\begin{thm} \label{thm:mahler-congruence_signed}
%    Let $I\subset R^{\pm}=\Z[\{\pm 1\}\times\Z^{\odd}]$ be an ideal and  $f\in \Q(x)$ with $I\subset  \Ann_{\Z[\{\pm 1\}\times\Z^{\odd}]}(f)$. 
Let $f\in \Q(x)$ and $\alpha \in \Z[\{\pm 1\}\times\Z^{\odd}]$. Suppose there exists $\beta\in \Ann_{\Z[\{\pm 1\}\times\Z^{\odd}]}(f)$ such that $\theta_d(\alpha)=\theta_d(\beta)$ and $f$ is defined at $\zeta_d^{(\epsilon,m)}$ for each $(\epsilon,m) \in \supp(\alpha)\cup \supp(\beta)$. Then $\Phi_d$ is a factor of $\alpha f$. 
\end{thm}
\begin{proof}
    Since $\beta\in \Ann_{\Z[\{\pm 1\}\times\Z^{\odd}]}(f)$, $\beta f=0$ and hence $\alpha f=(\alpha-\beta )f$ in $\Q(x)$. 
    Note that $\supp(\alpha-\beta)\subset \supp(\alpha)\cup \supp(\beta)$. Hence $f$ is defined at $\zeta_d^{(\epsilon,m)}$ for each $(\epsilon,m) \in \supp(\alpha-\beta)$. By Proposition~\ref{prop:0mod-rational_signed}, $\Phi_d$ is a factor of $(\alpha-\beta )f=\alpha f$.
\end{proof}

{Let us demonstrate \cref{thm:mahler-congruence_signed} by a concrete example. This example illustrates a cyclotomic factor that cannot be explained using the ordinary Mahler algebra $R=\Z[\Z^\circ]$, but is explained by the signed Mahler algebra $R^\pm$.

\begin{expl}
Let \(f(x)=\dfrac{x}{1+x^2}\). Then both $[(1,1)]-[(1,-1)]$ and     $[(1,1)]+[(-1,1)]$ belong to \(\Ann_{R^{\pm}}(f)\), since
\[
    f(x)-f(1/x)=0
    \qquad\text{and}\qquad
    f(x)+f(-x)=0.
\]
Consider
\[
    \psi_{35}=[(1,1)]-[(1,5)]-[(1,7)]+[(1,35)],
\]
where we view \(\psi_{35}\) as an element of \(R^\pm\) via the embedding
\([m]\mapsto [(1,m)]\). Let
\[
    \beta=[(1,1)]+[(-1,-1)]-[(1,5)]-[(-1,-5)].
\]
Then \(\beta\in \Ann_{R^\pm}(f)\). Moreover,
\[
    \theta_{24}(\psi_{35})
    =[1]-[5]-[7]+[11]
    =\theta_{24}(\beta).
\]
Consequently, \cref{thm:mahler-congruence_signed} implies that \(\Phi_{24}\) is a factor of $M_{35}\left(\dfrac{x}{x^2+1}\right).$
In fact, if we write
$M_{35}\left(\dfrac{x}{x^2+1}\right)=\frac{P}{Q}$ where $ \gcd(P,Q)=1$,
then \(P\) has the following factorization: \[P=\Phi_1\Phi_2\Phi_3\Phi_6\Phi_8\Phi_{24}G,\]
where \(G\) is an irreducible polynomial of degree \(48\). It is important to remark that if we
only consider the Mahler algebra \(R=\Z[\Z^\circ]\), then \(\Ann_R(f)\) contains
\([1]-[-1]\). Using \cref{pro:equivalence_02}, we can check that
\[
    \phi_{24}(\psi_{35})\notin \phi_{24}(I),
\]
where \(I\) is the ideal in \(R\) generated by \([1]-[-1]\). In other words, the Mahler algebra
\(R\) is not sufficient to explain the factor \(\Phi_{24}\).
\end{expl}
}

\begin{lem}
Let
\( \alpha=\sum\limits_{(\epsilon,m)\in \supp(\alpha)} a_{\epsilon,m}[(\epsilon,m)]\in R^\pm. \)
For each residue class $r\in (\mathbb Z/d\mathbb Z)^\circ$, define
\(
c_r=\sum\limits_{\substack{(\epsilon,m)\in \supp(\alpha)\\
\theta_d(\epsilon,m)=r}}
a_{\epsilon,m}.
\)
Then
\( \theta_d(\alpha)\in \theta_d(\langle [(1,1)]+[(-1,1)]\rangle) \) 
if and only if
\( c_r=c_{r+d_0} \)
for every residue class $r\in \mathbb (Z/d)^\circ$.
\end{lem}

\begin{proof}
    Clearly, we have $\theta_d(\alpha)=\sum c_r[r]$. Note also that for any $\beta\in R^\pm$, $\phi_d(\beta)= \sum_{r} m_r[r]$ is supported on odd residue classes. (Since $d$ is even, the notion of odd residue classes makes sense.) 
   Since $\phi_d([(1,1)]+[(-1,1)] )=[1]+[d_0+1]$ an arbitrary element in $\phi_d( \langle [(1,1)]+[(-1,1)] \rangle)$ is of the form
       \[
       \left(\sum_{r} m_r [r]\right) ([1] + [d_0+1]) = \sum_{r} m_r ([r] + [rd_0+r]) = \sum_{r} m_r ([r] + [d_0+r]).
       \]
       Here we note that for $r$ is odd then $rd_0+r\equiv d_0+r\pmod d$. 
       From this we see that $\phi_d(\alpha)\in \phi_d(\langle [(1,1)]+[(-1,1)] \rangle)$ if and only if $c_r=c_{r+d_0}$ for every residue class $r$.
\end{proof}

\begin{lem}
Let $\alpha\in R^\pm$. Suppose that $\theta_d(\alpha)\in \theta_d(\langle[(1,1)]+[(-1,1)] \rangle)$. Then there exists $\beta\in \langle [(1,1)]+[(-1,1)] \rangle$ such that $\theta_d(\alpha)=\theta_d(\beta)$ and $\supp(\beta)\subseteq \supp(\alpha)\cup(-\supp(\alpha))$. Here, $-\supp(\alpha)=\{(-\epsilon,a)\mid  (\epsilon,a)\in \supp(\alpha)\}.$
\end{lem}

\begin{proof}
Write
\(
\alpha=\sum\limits_{(\epsilon,m)\in \Supp(\alpha)}
a_{\epsilon,m}[(\epsilon,m)].
\)
For each residue class $r$, set
\(
c_r=\sum\limits_{\theta_d(\epsilon,m)=r}a_{\epsilon,m}.\) 
By the previous lemma,
\(
c_r=c_{r+d_0}.
\)

For each pair of residue classes \(\{r,r+d_0\}\)
with $c_r\neq 0$, we choose a representative \((\epsilon_r,m_r)\in \Supp(\alpha)\)
such that
\( \theta_d(\epsilon_r,m_r)=r.\)
Define \( \beta_r =c_r\left([(\epsilon_r,m_r)]+[(-\epsilon_r,m_r)]\right).\)
Then $\beta_r\in \langle [(1,1)]+[(-1,1)] \rangle$, and
\(
\theta_d(\beta_r) = c_r\big([r]+[r+d_0]\big). \)

Now set
\( \beta=\sum\limits_{\{r,r+d_0\}}\beta_r.\)
Then $\beta\in \langle [(1,1)]+[(-1,1)] \rangle$ and
\(
\theta_d(\beta)=\theta_d(\alpha).
\)
Moreover, by construction,
$\supp(\beta)\subseteq \supp(\alpha)\cup(-\supp(\alpha))$.
\end{proof}

\begin{thm} \label{thm:mahler-congruence_signed_01}
     Suppose $\alpha \in R^\pm$ such that $\theta_d(\alpha) \in \theta_d(\langle [(1,1)]+[(-1,1)] \rangle).$ Let $f \in \Q(x)$ such that $[(1,1)]+[(-1,1)]\in \Ann_{R^\pm}(f)$. If $f$ is defined at $\zeta_d^{(\epsilon,a)}$ for each $(\epsilon,a) \in \supp(\alpha)$, then $\Phi_d$ is a factor of $\alpha f$. 
\end{thm}
%\begin{rem}
%We remark that  if $f$ is odd; namely $I$ is the ideal generated by $[(1, 1)]+[(-1, 1)]$, then \cref{thm:mahler-congruence_signed} is the content of \cite[Corollary 4.2.2]{hyde2019polynomial}
%\end{rem}
\begin{proof}
 Since $I:=\langle [(1,1)]+[(-1,1)] \rangle \subset \Ann_{R^\pm}(f)$, one has $([(1,1)]+[(-1,1)])f=0$, i.e., $f(x)+f(-x)=0$. From this we see that if $f$ is defined at $\pm\zeta_d^m$  then $f$ is also defined at $\mp\zeta_d^{m}$. Hence $f$ is defined at $\zeta_d^{(\epsilon,a)}$ for each $(\epsilon,a)\in \supp(\alpha)\cup(-\supp(\alpha))$.

    By the previous lemma, there exists $\beta\in I$ such that $\theta_d(\beta)=\theta_d(\alpha)$ and
    \(
    \supp(\beta)\subseteq \supp(\alpha)\cup(-\supp(\alpha)).
    \)
    Clearly, \( \supp(\alpha)\cup\supp(\beta) \subseteq \supp(\alpha)\cup(-\supp(\alpha)).\) Hence $f$ is defined at $\zeta_d^{(\epsilon,a)}$ for each $(\epsilon,a)\in \supp(\alpha)\cup \supp(\beta)$. 
    By Theorem~\ref{thm:mahler-congruence_signed}, $\Phi_d$ is a factor of $\alpha f$. 
\end{proof}
We have a similar result for $[(1,1)-[(-1,1)]]$.

    \begin{thm} \label{thm:mahler-congruence_signed_02}
Suppose $\alpha \in R^\pm$ such that $\theta_d(\alpha) \in \theta_d(\langle [(1,1)]-[(-1,1)] \rangle).$ Let $f \in \Q(x)$ such that $[(1,1)]-[(-1,1)]\in \Ann_{R^\pm}(f)$. If $f$ is defined at $\zeta_d^{(\epsilon,a)}$ for each $(\epsilon,a) \in \supp(\alpha)$, then $\Phi_d$ is a factor of $\alpha f$. 
\end{thm}

\section{Mahler algebra, Galois symmetries, and cyclotomic factors of rational necklace functions} \label{sec:galois}
In this section, we explain an unexpected phenomenon of cyclotomic factors of $M_n(f)$. More precisely, in \cref{sec:mahler} and \cref{sec:signed-mahler}, our primary focus is on the symmetries of $f$ and the arithmetic of $n$. It turns out that there also exists a hidden interaction between these and the Galois group of $\Q(\zeta_d)/\Q$—which, in turn, is another source of cyclotomic factors for $M_n(f)$.

We motivate our discussion with two concrete examples of this new phenomenon where \cite[Theorem 2.19]{chidambaram2023fekete} and \cref{thm:mahler-congruence} and \cref{thm:mahler-congruence_signed} could not explain the appearance of some new cyclotomic factors. The first example was discovered while we were working on \cite{chidambaram2023fekete}. Recall that $\tilde{F}_n(x) = M_n(\frac{x}{1-x}).$  

\begin{prop} \label{prop:example-new-factors}
Let $p$ be a prime number.
\begin{enumerate}
\item  If $p\equiv 1\pmod 4$  then $\Phi_{24}$ is a factor of  $\widetilde{F}_{3\times 7\times p}$.
\item If $p\equiv 1\pmod {5}$  then $\Phi_{20}$ is a factor of $\widetilde{F}_{3\times 5\times p}$.
\end{enumerate}
\end{prop}
\begin{proof} (1) Let $\zeta=\zeta_{24}$. We have
\[
\begin{aligned}
\tilde{F}_{21}(\zeta)&=\frac{\zeta}{1-\zeta} -\frac{\zeta^3}{1-\zeta^3}-\frac{\zeta^7}{1-\zeta^7}+\frac{\zeta^{21}}{1-\zeta^{21}}
=\frac{\zeta(1+\zeta+\zeta^2)-\zeta^3}{1-\zeta^3} -\frac{\zeta^7(1+\zeta^7+\zeta^{14})-\zeta^{21}}{1-\zeta^{21}}\\
&=\frac{\zeta+\zeta^2}{1-\zeta^3} - \frac{\zeta^3(\zeta^7+\zeta^{14})}{\zeta^3(1-\zeta^{21})}
=\frac{\zeta+\zeta^2+\zeta^{10}+\zeta^{17}}{1-\zeta^3} 
=\frac{\zeta^{-2}(\zeta^3+\zeta^{12})+\zeta^2(1+\zeta^{15})}{1-\zeta^3}\\
&=\frac{\zeta^{-2}(\zeta^3-1)+\zeta^2(1-\zeta^{3})}{1-\zeta^3}=\zeta^2-\zeta^{-2}=\zeta^6.
\end{aligned}
\]
(Since $\Phi_{24}(x)=x^8-x^4+1$, one has $\zeta^2-\zeta^{-2}=\zeta^6$.)

Hence $\tilde{F}_{21p}(\zeta)=\tilde{F}_{21}(\zeta)-\tilde{F}_{21}(\zeta^p)=\zeta^6-\zeta^{6p}=\zeta^6-\zeta^6=0$.

(2) Let $\zeta=\zeta_{20}$. Note that $\zeta^{10}=-1$. We have
\[
\begin{aligned}
\tilde{F}_{15}(\zeta)&=\frac{\zeta}{1-\zeta} -\frac{\zeta^3}{1-\zeta^3}-\frac{\zeta^5}{1-\zeta^5}+\frac{\zeta^{15}}{1-\zeta^{15}}\\
&=\frac{\zeta(1+\zeta+\zeta^2+\zeta^3+\zeta^4)-\zeta^5}{1-\zeta^5} -\frac{\zeta^3(1+\zeta^3+\zeta^{6}+\zeta^9+\zeta^{12})-\zeta^{15}}{1-\zeta^{15}}\\
&=\frac{\zeta+\zeta^2+\zeta^3+\zeta^4}{1-\zeta^3} - \frac{\zeta^5(\zeta^3+\zeta^6+\zeta^9+\zeta^{12})}{\zeta^5(1-\zeta^{15})}\\
&=\frac{\zeta+\zeta^2+\zeta^3+\zeta^4+\zeta^8+\zeta^{11}+\zeta^{14}+\zeta^{17}}{1-\zeta^5}\\
&=(\frac{\zeta+\zeta^{11})+(\zeta^4+\zeta^{14})+( \zeta^3+ \zeta^8)+(\zeta^2+ \zeta^{17})}{1-\zeta^5}\\
&=\frac{\zeta^{-2}(\zeta^5+\zeta^{10})+\zeta^2(1+\zeta^{15})}{1-\zeta^{5}}=\zeta^2-\zeta^{-2} = \zeta_5^2 - \zeta_5^3.
\end{aligned}
 \]

Hence $\tilde{F}_{15p}(\zeta)=\tilde{F}_{15}(\zeta)-\tilde{F}_{15}(\zeta^p)=\zeta^2-\zeta^{2p}-(\zeta^{-2}-\zeta^{-2p})=0$.
\end{proof}

Our goal is to generalize \cref{prop:example-new-factors} to a broader class of rational functions. While the explicit calculations provided in \cref{prop:example-new-factors} can, in principle, be applied to any given $f \in \mathbb{Q}(x)$, it is unclear how to extend this approach to a \textit{generic} $f$. Fortunately, we can overcome this issue using a new tool: Galois theory. We discuss a second example that aims to streamline the argument described in \cref{prop:example-new-factors}. We begin with the following.% which is very important in our arguments.
\begin{lem}\label{lem-defined}
Let $f(x) \in \mathbb{Q}(x)$, $d>1$ a positive integer and $a\in \Z$. 
Suppose that $f$ is defined at $\zeta_d^a$. 
Then $f$ is also defined at $\zeta_d^{at}$ and at $\zeta_d^{a/q}$, 
for any integer $t$ with $\gcd(t,d)=1$ and any integer $q \mid a$ with $\gcd(q,d)=1$.
\end{lem}
\begin{proof}
    Write $f(x) = \frac{g(x)}{h(x)}$ with $g(x), h(x) \in \mathbb{Q}[x]$ and $\gcd(g,h)=1$. 
By assumption, $h(\zeta_d^a) \neq 0$. 

Let $t$ be an integer with $\gcd(t,d)=1.$ Then the map $\zeta_d \mapsto \zeta_d^t$ 
extends to a field automorphism $\sigma_t$ of $\mathbb{Q}(\zeta_d).$
Hence,  we have
\[
h(\zeta_d^{at}) = \sigma_t\big(h(\zeta_d^a)\big).
\]
 Since $h(\zeta_d^a) \neq 0$, 
it follows that $h(\zeta_d^{at}) \neq 0$, and thus $f$ is defined at $\zeta_d^{at}$.

Now let $q \mid a$ with $\gcd(q,d)=1$. Choose an integer $l$ such that 
$lq \equiv 1 \pmod{d}$. Then $a/q\equiv al\pmod d$ and hence $\zeta_d^{a/q} = \zeta_d^{al}$. Since $\gcd(l,d)=1$, 
the previous argument shows that $h(\zeta_d^{a/q}) \neq 0$. Therefore, 
$f$ is defined at $\zeta_d^{a/q}$ as desired.
\end{proof}

\begin{prop} \label{prop:Galois-toy-example}
     Let $f(x)\in \Q(x)$ be  such that $f(x)+f(1/x)=C$ is constant and $f(x)$ is defined at $\zeta_{24}^a$, for $a\in \{1,3,7,21\}$. Then, the following properties hold.
     \begin{enumerate}
  \item[(a)] If $p\equiv 1\pmod 8$ is a prime then $\Phi_{24}\mid M_{3\times 7\times p}(f)$.
  \item[(b)] Let $g(x)=f(x)-f(-x)$. Assume further that $f(x)$ is defined at $-\zeta_{24}^a$, for $a\in \{1,3,7,21\}$ and that $M_{3\times 7}(g)(\zeta_{24})=0$.  If $p\equiv 1\pmod 4$  then $\Phi_{24}\mid M_{3\times 7\times p}(f)$.
  \end{enumerate}
\end{prop}
\begin{proof} Let $\zeta=\zeta_{24}$. By our assumption, $M_{3\times7}(f)$ is defined at $\zeta$. We will show that $M_{3\times 7}(f)(\zeta)\in \Q(\zeta_8)$. 

We know that $G=\Gal(\Q(\zeta_{24})/\Q)=(\Z/24\Z)^\times=\{\overline{1},\overline{5},\overline{7},\overline{11},\overline{13},\overline{17},\overline{19},\overline{23}\}$. We denote the Galois action of $\overline{a}\in G$ on an element $u\in \Q(\zeta)$ by $\overline{a}\cdot u$. Recall that this action is given by $\overline{a}\cdot \zeta=\zeta^a$. Let $H=\{\overline{1},\overline{17}\}$ be a subgroup of $G$. 
We have
\[
\begin{aligned}
\overline{17}\cdot M_{21}(f)(\zeta)&=M_{21}(f)(\zeta^{17})=f(\zeta^{17})-f(\zeta^{3\cdot 17})-f(\zeta^{7\cdot 17})+f(\zeta^{21\cdot 17})\\
&=f(\zeta^{17})-f(\zeta^{3})-f(\zeta^{23})+f(\zeta^{21})\\
&=f(\zeta^{1})-f(\zeta^{3})-f(\zeta^{7})+f(\zeta^{21}).
\end{aligned}
\]
The last equality follows from $f(\zeta^{17})+f(\zeta^{7})=f(\zeta^{23})+f(\zeta^{1})$.  (Note also that by Lemma~\ref{lem-defined}$, f$ is defined at $\zeta^{17a}$, for $a=1,3,7,21$.) Thus $M_{21}(f)(\zeta)\in \Q(\zeta)^H= \Q(\zeta_{8})$.

(a) We can write $M_{21}(f)(\zeta)= P(\zeta_8)$, for some $P(x)\in \Q[x]$. Suppose that $p\equiv 1 \pmod 8$ is a prime. Then by Lemma~\ref{lem-defined}, $M_{21}(f)$ is defined at $\zeta^p$, and we have 
\[M_{21}f(\zeta^p)=\overline{p} \cdot M_{21}(f)(\zeta) = \overline{p}\cdot P(\zeta_8)=P(\zeta_8^p)=P(\zeta_8).\]
Hence $M_{21p}(f)(\zeta)=M_{21}(f)(\zeta)-M_{21}(f)(\zeta^p)=0$.

(b) Since $M_{3\times 7}(g)(\zeta_{24})=0$, we see that
    \[
    M_{21}(f)(\zeta)=M_{21}(f)(-\zeta)=M_{21}(f)(\zeta^{13})= \overline{13} \cdot M_{21}(f)(\zeta).
    \]
    Hence $M_{21}(f)(\zeta)$ is also fixed by $\overline{13}\in G=\Gal(\Q(\zeta_{24})/\Q)$. Thus $M_{21}(f)(\zeta)$ is fixed by the subgroup $H'=\langle \overline{13},\overline{17}\rangle =\{\overline{1},\overline{5},\overline{13},\overline{17}\}$. This implies that $M_{21}(f)(\zeta)\in Q(\zeta)^{H'}=\Q(\zeta_4)$.  We can write $M_{21}(f)(\zeta)= P(\zeta_4)$, for some $P(x)\in \Q[x]$. Suppose that $p\equiv 1 \pmod 4$ is a prime. Then by Lemma~\ref{lem-defined}, $M_{21}(f)$ is defined at $\zeta^p$, and we have 
\[M_{21}f(\zeta^p)=\overline{p} \cdot M_{21}(f)(\zeta) = \overline{p}\cdot P(\zeta_4)=P(\zeta_4^p)=P(\zeta_4).\]
Hence $M_{21p}(f)(\zeta)=M_{21}(f)(\zeta)-M_{21}(f)(\zeta^p)=0$.
\end{proof}

\begin{rem}
    We return to Example \ref{ex:3x5x7}. Let $f(x)=\dfrac{1-x}{1+x}$ and $\zeta=\zeta_{24}$.  We have
\[
g(x)=f(x)-f(-x)
=
\frac{1-x}{1+x}-\frac{1+x}{1-x}
=
\dfrac{-4x}{1-x^2},
\]
and
\[
M_{21}(g)(\zeta)
=
g(\zeta)-g(\zeta^3)-g(\zeta^7)+g(\zeta^{21}).
\]

We first compute
\[
g(x)-g(y)
=
\frac{-4x}{1-x^2}-\frac{-4y}{1-y^2}
=
\frac{4(y-x)(1+xy)}
{(1-x^2)(1-y^2)}.
\]
Hence
\[
g(\zeta)-g(\zeta^3)
=
\frac{4(\zeta^3-\zeta)(1+\zeta^4)}
{(1-\zeta^2)(1-\zeta^6)}= -\frac{4\zeta(1+\zeta^4)}
{(1-\zeta^6)}.
\]
Similarly,
\[
g(\zeta^7)-g(\zeta^{21})
=
\frac{4(\zeta^{21}-\zeta^7)(1+\zeta^{28})}
{(1-\zeta^{14})(1-\zeta^{42})}=-\frac{4\zeta^7(1+\zeta^{28})}
{(1-\zeta^{42})}.
\]

Using \(\zeta^{24}=1\), 
we have
\[
g(\zeta^7)-g(\zeta^{21})
=- \frac{4\zeta^7(1+\zeta^{28})}
{(1-\zeta^{42})} = -4\frac{\zeta^6}{1-\zeta^{18}} \zeta(1+\zeta^4)
=-4\frac{1}{1-\zeta^{6}} \zeta(1+\zeta^4)=
g(\zeta)-g(\zeta^3).
\]
Here we use $\zeta^{12}=-1$ to obtain $\frac{\zeta^6}{1-\zeta^{18}} =\frac{\zeta^{12}}{\zeta^6-\zeta^{24}}=\frac{-1}{\zeta^6-1}=\frac{1}{1-\zeta^6}$.
Consequently,
\[
M_{21}(g)(\zeta)
=
\bigl(g(\zeta)-g(\zeta^3)\bigr)
-
\bigl(g(\zeta^7)-g(\zeta^{21})\bigr)
=
0.
\]

Proposition~\ref{prop:Galois-toy-example} implies that $\Phi_{24}$ is a factor of $M_{105}(f)$.
\end{rem}

We now formalize the argument for \cref{prop:Galois-toy-example} in a more general setting. 
\begin{thm} 
\label{thm:Galois}
Let \(d>1\)  and let \(d_1\) be a divisor of \(d\). 
Let
\[
H=\ker( (\Z/d)^{\times} \to (\Z/d_1)^{\times} ).
\]
Let $f(x)\in \Q(x)$ be such that $[1]+[-1] \in \Ann_{\Z[\Z^{\circ}]}(f)$. 
  Let \(\alpha=\alpha_1\alpha_2\in \mathbb Z[\mathbb Z^\circ]\). Suppose that:

\begin{enumerate}
    \item $f$ is defined at each $\zeta_{d}^{a}$ for each $a \in \supp(\alpha_1)$;
    \item For all $h \in \Z$ with $\overline{h}\in H$, we have $\phi_{d}(([h]-[1])\alpha_{1}) \in \phi_d(\langle[1]+[-1]\rangle)$;
    \item For each $a \in \supp(\alpha_2)$, $\gcd(a,d)=1$;
 \end{enumerate}
The following statements hold.
\begin{enumerate} 
    \item[(a)]  If $\alpha_2 \equiv 0 \mod{[d_1]}$ then $\Phi_d$ is a factor of $\alpha f.$
    \item[(b)]  Let \( g(x)=f(x)-f(-x).\) Suppose further that $d_1=4d_0$ is a multiple of 4 and $d/d_1$ is odd and $\alpha_1g(\zeta_d)=0$. If  $\alpha_2 \equiv 0 \mod{[d_1/2]}$ then $\Phi_d$ is a factor of $\alpha f.$
\end{enumerate}
 
\end{thm}

\begin{proof}
    % First, we may write $\alpha_2=\sum_{k \in S} c_k([a_k]-[b_k]),$ where $S$ is the set such that $a_k , b_k \in \Supp(\alpha_2),$ $a_k\equiv b_k\pmod{d_1}$ and $\gcd(a_k,d)=\gcd(b_k,d)=1$for all $k \in S$. 
Let $\zeta=\zeta_d$. By our assumption, $\alpha_1 f$ is defined at $\zeta$. We will show that $\alpha_1 f(\zeta)\in \Q(\zeta_{d_1})$. 

We know that $G=\Gal(\Q(\zeta)/\Q)=(\Z/d\Z)^\times = \{\overline{a} \mid 1\leq a < d , \gcd(a,d) =1\}.$ We denote the Galois action of $\overline{a}\in G$ on an element $u\in \Q(\zeta)$ by $\overline{a}\cdot u$. Recall that this action is given by $\overline{a}\cdot \zeta =\zeta^a$.  Here, we make an important remark about the relationship between the Galois action and the Mahler action. Namely; if $\gcd(a, m)=1$, and $p \in \Q(x)$  which is defined at $\zeta_m$, then these actions are compatible. More precisely, we have $\bar{a} p(\zeta_m) =([a] \cdot p )(\zeta_m)$ since they are both equal to $p(\zeta_m^a).$

By the assumption that $f$ is defined at $\zeta^a$ for each $a\in \supp(\alpha_1)$ and by \cref{lem-defined}, we see that $f$ is defined at each $\zeta^{b}$ for each $ b \in \Supp(([h]-[1])\alpha_1), \overline{h} \in H.$ For each $\overline{h}\in H$,  we have
\[
\overline{h}\cdot \alpha_1 (f)(\zeta)=\alpha_1 (f)(\zeta^{h})
=([h]\alpha_1) (f)(\zeta)
=\alpha_1 (f)(\zeta).
\]
 The last equality follows from $([h]-[1])\alpha_1 f(\zeta) = 0,$ by \cref{thm:mahler-congruence}. Thus $\alpha_1(f)(\zeta)\in \Q(\zeta)^{H}= \Q(\zeta_{d_1})$.
 Indeed, $H$ is precisely the subgroup fixing $\zeta_{d_1}=\zeta^{d/d_1}$.

 (a)  We can write $\alpha_1(f)(\zeta)= P(\zeta_{d_1})$, for some $P(x)\in \Q[x]$. 
  Write $\alpha_2=\sum\limits_{a\in S} n_a[a]$, where $S$ is the support of $\alpha_2$. By the assumption that $\gcd(a,d)=1$ for each $a\in S$ and  by \cref{lem-defined}, we see that $f$ is defined at each $\zeta^{b}$ for each $b\in \supp(\alpha_2\alpha_1)$. Then 
\[\begin{aligned}
    \alpha f(\zeta) &= \alpha_2 \alpha_1f(\zeta)= \sum_{a \in S} n_a[a]\alpha_1f(\zeta) = \sum_{a \in S} n_a\alpha_1 f(\zeta^{a})=\sum_{a \in S} n_a \overline{a}\cdot \alpha_1 f(\zeta)\\
    &=\sum_{a \in S} n_a \overline{a} \cdot P(\zeta_{d_1})
    = \sum_{a \in S} n_a [a]P(\zeta_{d_1}) =\alpha_2 P(\zeta_{d_1}).
 \end{aligned}\]
 Since $\alpha_2\equiv 0\mod {[d_1]}$, by Proposition~\ref{prop:0mod-rational} we have $\alpha_2 P(\zeta_{d_1})=0$. Hence $\alpha f(\zeta)=0$ and we are done.

 (b) Now we use the additional hypothesis involving \(     g(x)=f(x)-f(-x).
\)
Since \(d\) is divisible by \(4\), the integer \(d/2+1\) is coprime to \(d\), and
\(
    -\zeta=\zeta^{d/2+1}.
\) 
By assumption,
\(
    \alpha_1 g(\zeta)=0.
\)
Thus
\[
    0     =     \alpha_1 f(\zeta)-\alpha_1 f(-\zeta)     =     \alpha_1 f(\zeta)-\alpha_1 f(\zeta^{d/2+1}) = \alpha_1 f(\zeta) -\overline{(\frac{d}{2}+1)} \cdot \alpha_1 f(\zeta).
\]
Therefore \(\alpha_1 f(\zeta)\) is fixed by the Galois automorphism \(
    \zeta\longmapsto \zeta^{d/2+1}.\)
Consequently,
\[
    \alpha_1 f(\zeta)
    \in
    \mathbb Q(\zeta_{d_1})\cap
    \mathbb Q(\zeta_d)^{\langle d/2+1\rangle}.
\]
Since the automorphism \(\zeta_d\mapsto -\zeta_d\) restricts on
\(\mathbb Q(\zeta_{4d_0})\) to the automorphism
\[
    \zeta_{4d_0}\longmapsto -\zeta_{4d_0},
\]
its fixed field inside \(\mathbb Q(\zeta_{4d_0})\) is \(     \mathbb Q(\zeta_{2d_0}).\)
Hence
\[
    \alpha_1 f(\zeta)\in \mathbb Q(\zeta_{2d_0}).
\]
Thus there exists a polynomial \(P(x)\in \mathbb Q[x]\) such that
\(
    \alpha_1 f(\zeta)=P(\zeta_{2d_0}).
\)

Now write
\(
    \alpha_2=\sum\limits_{a\in \operatorname{Supp}(\alpha_2)} n_a[a].
\)
By the assumption that $\gcd(a,d)=1$ for each $a\in S$ and  by \cref{lem-defined}, we see that $f$ is defined at each $\zeta^{b}$ for each $b\in \supp(\alpha_2\alpha_1)$. Then 
\[\begin{aligned}
    \alpha f(\zeta) &= \alpha_2 \alpha_1f(\zeta)= \sum_{a \in S} n_a[a]\alpha_1f(\zeta) = \sum_{a \in S} n_a\alpha_1 f(\zeta^{a})=\sum_{a \in S} n_a \overline{a}\cdot \alpha_1 f(\zeta)\\
    &=\sum_{a \in S} n_a \overline{a} \cdot P(\zeta_{2d_0})
    = \sum_{a \in S} n_a [a]P(\zeta_{2d_0}) =\alpha_2 P(\zeta_{2d_0}).
 \end{aligned}\]
 Since $\alpha_2\equiv 0\mod {[2d_0]}$, by Proposition~\ref{prop:0mod-rational} we have $\alpha_2 P(\zeta_{2d_0})=0$. Hence $\alpha f(\zeta)=0$ and we are done.
\end{proof}

\begin{expl}
    We consider the case in \cref{prop:Galois-toy-example}: \[
d = 24, \quad d_1 = 8, \quad \alpha_1 = \psi_{21}, \quad \alpha_2 = \psi_p,
\]
where $p \equiv 1 \pmod{8}$. Let $f \in \mathbb{Q}(x)$ be a rational function satisfying $
f(x) + f(1/x) = 0,
$
and assume that $f$ is defined at each $\zeta_{24}^a$ for $a \in \Supp(\psi_{21})$. %Let \[I = \langle [1] + [-1] \rangle_{\mathbb{Z}[\mathbb{Z}^\circ]}.\]
In this case, one computes that
\[
H = \{\bar{1}, \overline{17}\}.
\]
Moreover,
\begin{align*}
\phi_{24}\bigl(([17]-[1])\psi_{21}\bigr)
&= \phi_{24}\bigl(([17]-[1])([1] - [3] - [7] + [21])\bigr) \\
&= [\overline{17}] + [\overline{7}] - \bigl([\overline{23}] + [\overline{1}]\bigr) =([\overline{7}]-[\overline{1}])([\overline{1}]+[\overline{-1}]) \in \phi_{24}(\langle [1] + [-1] \rangle).
\end{align*}
On the other hand, since $p \equiv 1 \pmod{8}$, we have
\[
\psi_p = [1] - [p] \equiv 0 \mod{[8]}.
\]
Then, we have $\Phi_{24}$ is a factor of $\alpha f = \psi_{3 \times 7 \times p} f,$ as proved in \cref{prop:Galois-toy-example}.
\end{expl}

\appendix
\section{Some ring-theoretic properties of the Mahler algebra $\Z[\Z^{\circ}]$}
In this section, we study some classes of rational functions in $\Q(x)$ with extra symmetries. Here, we make a rather simple observation that if $\alpha, \beta \in R:=\Z[\Z^{\circ}]$ such that $\alpha \beta=0$, then $\alpha \in \Ann_{R}(\beta f)$ for all $f \in \Q(x).$ In other words, the existence of zero divisors in $\Z[\Z^{\circ}]$ provides a good source of rational functions with extra symmetries. Examples of such $\alpha, \beta$ include $[1]+[-1]$ and $[1]-[-1]$ since 
\[ ([1]+[-1])([1]-[-1])=[1]-[1]=0.\]
It is natural to ask whether it is possible to classify all zero divisors in $\Z[\Z^{\circ}]$. The following theorem gives a complete answer to this question. 

\begin{thm}\label{thm:zero_divisors} 
Let $\alpha$ be a zero-divisor in $\Z[\Z^{\circ}]$. Then, $\alpha$ is either a multiple of $[1]+[-1]$ or $[1]-[-1].$
\end{thm}

\begin{proof}
        First, we remark that $\Z^\circ$. $\Z^\circ = \Z ~\backslash~\{0\}$ can be identified as a direct product of two simpler monoids:
        \begin{itemize}
            \item The group $G = \{1, -1\}$;
            \item The monoid of positive integers $\N = \Z^+ = \{1, 2, 3, \dots\}$.
        \end{itemize}
        Every $m \in \Z^\circ$ has a unique representation as $m = s \times p$ with $s \in G$ and $p \in N$. Thus, we have a monoid isomorphism $\Z^\circ = G \times N$. Hence, we can rewrite \[
            R = \Z[\Z^\circ] \cong \Z[G \times N] \cong (\Z[N])[G].
        \]
        % From the book Robert Gilmer - Commutative Semigroup Ring, Theorem 8.1
        Now we define $S = \Z[N] = \Z[\Z^+]$, then our ring $R$ is isomorphic to the group ring $S[G] = S[\{-1, 1\}]$. Note that since the monoid $N = \Z^+$ is cancellative (that is if $ab = ac$ then $b = c$) and torsion-free (that is the only element $n$ with $n^k = 1$ is $n = 1$). A result by Gilmer \cite{gilmer1984commutative} stated that if a ring $U$ is an integral domain and a monoid $V$ is cancellative and torsion-free, then the monoid ring $U[V]$ is also an integral domain. Since our $\Z$ is an integral domain and $N$ is cancellative and torsion-free, $\Z[N]$ is an integral domain, thus $\Z[N]$ has no zero divisors.
        \\
        For the group ring $R$, let $g$ be the element of $S[G]$ corresponding to $[-1]$, thus $g^2 = [1]$. Any $\alpha \in R$ can be uniquely written as $\alpha = A + Bg$ with $A, B \in S$ and $A$ represents the part of $\alpha$ with positive indices while $B$ represents the part with negative indices. Specifically, for $\alpha = \sum_{m \in Z^\circ}a_m[m]$, then:
        \begin{itemize}
            \item $A = \sum_{p \in \Z^+}a_p[p]$;
            \item $B = \sum_{p \in \Z^+}a_{-p}[p]$.
        \end{itemize}
        An element $\alpha = A + Bg$ is a zero divisor if there exists a non-zero element $\beta = C + Dg$, with $C, D \in S$ such that $\alpha\beta = 0$. This is equivalent to \[
            (A + Bg)(C + Dg) = (AC + BD) + (AD + BC)g = 0.
        \]
        Since we are calculating in the ring $S[G]$, then it gives us the system of equations
        \[
        \begin{cases}
            AC + BD = 0\\
            AD + BC = 0.\\
        \end{cases}
        \]
        For this system to have non-zero solution $C, D$, the determinant of the coefficient matrix must be zero: \[
            \det\left(\begin{matrix}
                    A & B \\
                    B & A
                \end{matrix}\right) = A^2 - B^2 = 0.
        \]
        $S$ is an integral domain, then $A^2 - B^2 = 0$ implies $(A - B)(A + B) = 0$, which implies that either $A - B = 0$ or $A + B = 0$. We now consider cases:
        \begin{enumerate}
            \item $A = B$, then the element $\alpha = A(1 + g)$, with $A$ is non-zero. This is a zero divisor because it is annihilated by the non-zero element $1 - g$: $\alpha(1 - g) = A(1 + g)(1 - g) = A(1 - g^2) = A(1 - 1) = 0$. Since $A = B$, it means $\sum_{p \in \Z^+}a_p[p] = \sum_{p \in \Z^+}a_{-p}[p]$. Since the elements $[p]$ form a basis for $S$, the coefficients must be equal for each $p$, that is $a_p = a_{-p}$ for all $p \in \Z^+$. We call these ``even'' elements.
            \item $A = -B$, then the element $\alpha = A(1 - g)$, with $A$ is non-zero. This is a zero divisor because it is annihilated by the non-zero element $1 + g$: $\alpha(1 + g) = A(1 - g)(1 + g) = A(1 - g^2) = A(1 - 1) = 0$. Since $A = -B$, it means $\sum_{p \in \Z^+}a_p[p] = -\sum_{p \in \Z^+}a_{-p}[p]$. Since the elements $[p]$ form a basis for $S$, the coefficients must be equal for each $p$, that is $a_p = -a_{-p}$ for all $p \in \Z^+$. We call these ``odd'' elements
        \end{enumerate}
        Thus, the set of zero divisors in $\Z[\Z^\circ]$ is the union of two sets of non-zero elements: those whose coefficients have even symmetry ($a_p = a_{-p}$) and those whose coefficients have odd symmetry ($a_p = -a_{-p}$). Furthermore, this also means the set of all zero divisors is the union of the two principal ideals $([1] + [-1])$ and $([1] - [-1])$, excluding the zero element.
    \end{proof}

We discuss a related problem about zero divisors. As we mentioned earlier, if $\alpha \beta = 0$, then $\alpha \in \Ann_{R}(\beta f)$ for all $f \in \mathbb{Q}(x)$. We might wonder whether the converse is true: specifically, if $\alpha g = 0$, does it follow that $g = \beta f$ for some $f \in \mathbb{Q}(x)$? This question can be neatly framed within module theory using the idea of an exact sequence.

 \begin{definition}
A sequence of \( R \)-modules and homomorphisms
\[
\cdots \to M_{i-1} \xrightarrow{f_{i-1}} M_i \xrightarrow{f_i} M_{i+1} \to \cdots
\]
is called exact if, for every integer \( i \), the following condition holds:
\[
\operatorname{Im}(f_{i-1}) = \ker(f_i).
\]
\end{definition}
%We can not state our proposition. 
\begin{prop}
Let $\alpha = [1]-[-1]$ and $\beta = [1]+[-1].$ Then, the following sequence is exact 
    \[ M \xrightarrow[]{\alpha} M  \xrightarrow[]{\beta} M \xrightarrow[]{\alpha} M. \]

    Here $\alpha$ (respectively $\beta$) is the multiplication by $\alpha$ (respectively $\beta$). 
\end{prop}

\begin{proof}
        The given sequence is exact if and only if $\operatorname{Im}(\alpha) = \ker(\beta)$ and $\operatorname{Im}(\beta) = \ker(\alpha)$. We only  prove that $\operatorname{Im}(\alpha) = \ker(\beta)$. The other set equality can be proved in a similar way.  First, we will identify $\operatorname{Im}(\alpha)$ and $\ker(\beta)$.
        \begin{enumerate}
            \item $\operatorname{Im}(\alpha)$: The image of $\alpha$ is the set of all $g(x) \in M$ such that $g(x) = \alpha \cdot h(x)$ for some $h(x) \in M$:
                  \begin{itemize}
                      \item $\alpha \cdot h(x) = ([1] - [-1])\cdot h(x) = h(x^1) - h(x^{-1}) = h(x) - h(1/x)$.
                      \item Then, $\operatorname{Im}(\alpha) = \{h(x) - h(1/x): h(x) \in \Q(x)\}$
                  \end{itemize}
            \item $\ker(\beta)$: The kernel of $\beta$ is the set of all $f(x)$ such that $\beta \cdot f(x) = 0$.
                  \begin{itemize}
                      \item $\beta \cdot f(x) = ([1] + [-1])\cdot f(x) = f(x^1) + f(x^{-1}) = f(x) + f(1/x)$.
                      \item Then, $\ker(\beta) = \{f(x) \in \Q(x): f(x) + f(1/x) = 0\}$
                  \end{itemize}
        \end{enumerate}
        We will show that $\operatorname{Im}(\alpha) = \ker(\beta)$:
        \begin{enumerate}
            \item $\operatorname{Im}(\alpha) \subseteq \ker(\beta)$: Consider $g(x) \in \operatorname{Im}(\alpha)$, then $g(x) = h(x) - h(1/x)$ for $h(x) \in \Q(x)$. We will check if $\beta \cdot g(x) = 0$. $\beta \cdot g(x) = g(x) + g(1/x)$. Note that we also have $g(1/x) = h(1/x) - h(x)$, then $\beta \cdot g(x) = (h(x) - h(1/x)) + (h(1/x) - h(x)) = 0$, or $g(x) \in \ker(\beta)$.
            \item $\ker(\beta) \subseteq \operatorname{Im}(\alpha)$: Consider $g(x) \in \ker(\beta)$, then $g(x) + g(1/x) = 0$. Take $h(x) = \frac{g(x)}{2}$, then $h(x) - h(1/x) = \frac{g(x)}{2} - \frac{g(1/x)}{2}$. Since $g(x) + g(1/x) = 0$, it means $g(x) = -g(1/x)$, thus $h(x) - h(1/x) = \frac{g(x)}{2} + \frac{g(x)}{2} = g(x)$, thus $g(x) \in \operatorname{Im}(\alpha)$, then, $g(x) \in \operatorname{Im}(\alpha)$.
        \end{enumerate}
        These implies $\operatorname{Im}(\alpha) \subseteq \ker(\beta)$ and $\ker(\beta) \subseteq \operatorname{Im}(\alpha)$, that is $\operatorname{Im}(\alpha) = \ker(\beta)$. %, hence the sequence is exact.
    \end{proof}

\begin{lem} Let $\alpha\in \Z[\N^\circ]$ and $f\in \Q(x)\setminus \Q$. If $\alpha f=0$ then  $\alpha=0$.
    
\end{lem}
\begin{proof}
 Write
\( 
\alpha=\sum_{m\in S} a_m[m],
\)
where \(S=\supp(\alpha)\subset \mathbb N^\circ\) is finite and \(a_m\neq 0\) for all \(m\in S\).
We will show that \(\alpha f=0\) implies \(S=\varnothing\), hence \(\alpha=0\).

Write \(f=P/Q\), where \(P,Q\in \mathbb Q[x]\) are coprime and \(Q\neq 0\). Since \(f\notin \mathbb Q\), at least one of \(P,Q\) has positive degree. 

Suppose
\[
0=\alpha f(x)=\sum_{m\in S} a_m f(x^m)=\sum_{m\in S} a_m \frac{P(x^m)}{Q(x^m)}.
\]
Multiplying by
\( 
\prod\limits_{m\in S} Q(x^m),
\)
we get
\[
\sum_{m\in S}
a_m P(x^m)\prod_{\substack{n\in S\\ n\neq m}} Q(x^n)=0.
\]

 The summand corresponding to  \(m\) has a degree
\[
d_m:=m\deg P+\sum_{\substack{n\in S\\ n\neq m}} n\deg Q =\sum_{n\in S} n\deg Q + m(\deg P-\deg Q).
\]
Hence for $m,m'\in S$, we have
\[
d_m-d_{m'}=(m-m')(\deg P-\deg Q). 
\]
Suppose $\deg P\not=\deg Q$ then all  degrees $d_m$ are distinct. This forces $S$ is empty and $\alpha=0$.

 Assume now that \(\deg P=\deg Q\). Let
\(
c=\dfrac{\operatorname{lc}(P)}{\operatorname{lc}(Q)}
\)
and set
\(
R=P-cQ.
\) 
Then
\(
f=c+\frac{R}{Q}=c+g,
\)
and
\(
\deg R<\deg Q.
\).

Since \(\alpha f=0\), we have
\( 
\alpha\!\left(\frac{R}{Q}\right)
=
-c\,\alpha(1).
\)
The right-hand side is a constant, denoted $C\in \Q$. We obtain
\[
\sum_{m\in S} a_m \frac{R(x^m)}{Q(x^m)}=C.
\]

Multiplying by \(\prod_{m\in S}Q(x^m)\) yields
\[
\sum_{m\in S}
a_m R(x^m)
\prod_{\substack{n\in S\\ n\neq m}}
Q(x^n)
=
C\prod_{m\in S}Q(x^m).
\]

If $C\not=0$ then the degree of the right-hand side is
\(
\sum\limits_{m\in S} m\deg Q.
\)
On the other hand, each summand on the left has degree
\[
m\deg R+
\sum_{\substack{n\in S\\ n\neq m}}
n\deg Q
=
\sum_{n\in S} n\deg Q
-
m(\deg Q-\deg R).
\]
Since \(\deg R<\deg Q\), every summand on the left has degree strictly less than
\(\sum_{n\in S} n\deg Q\). Therefore \(C=0\), and hence
\(
\alpha g=0
\), where $g=\dfrac{R}{Q}$ with \(\deg R<\deg Q\). Note that  since $f$ is not a constant $g$ is also not a constant.
Thus we reduce to the previous case and hence  $\alpha=0$.
\end{proof}

\begin{cor} 
Let $f\in \Q(x)\setminus\Q $ and $\alpha\in \Z[\Z^\circ]$. If $\alpha f=0$ then  $\alpha$ is either a multiple of $[1]+[-1]$ or $[1]-[-1].$
    \end{cor}
\begin{proof}
Write
\(
\alpha=A+[-1]B,
\)
where \(A,B\in \mathbb Z[\mathbb N^\circ]\).

Define
\[
f_{+}(x)=\dfrac{f(x)+f(1/x)}{2}, \qquad
f_{-}(x)=\dfrac{f(x)-f(1/x)}{2}.
\]
Note that $f_{+}(1/x)=f_+(x)$ and $f_{-}(1/x)=-f_{-}(x)$, i.e, $[-1]f_{+}=f_{+}$, and  $[-1]f_{-}=-f_{-}$.

 We have
\[
0=\alpha f=\alpha f_{+}+\alpha f_{-}=(A+[-1]B)f_{+} +(A+[-1]B)f_{-}=(A+B)f_{+}+(A-B)f_{-}.
\]

Applying \([-1]\) to this equation gives
\[
(A+B)f_{+}-(A-B)f_{-}=0.
\]
Adding and subtracting the two equations yields
\[
(A+B)f_{+}=0,
\qquad
(A-B)f_{-}=0.
\]

Since \(f\notin \mathbb Q\), at least one of \(f_{+}\) and \(f_{-}\) is {non-constant}. By Lemma~A.4, if
\(f_{+}\) is non-constant, then \(A+B=0\), and therefore
\[
\alpha=A+[-1]B=A-[-1]A=A([1]-[-1]).
\]
Similarly, if \(f_{-}\) is non-constant, then \(A-B=0\), and therefore
\[
\alpha=A+[-1]A=A([1]+[-1]).
\]

Thus \(\alpha\) is a multiple of either \([1]-[-1]\) or \([1]+[-1]\).
\end{proof}

Let $\epsilon$ be the homomorphism $\epsilon\colon \Z[\N^\circ]\to \Z$, which sends $\alpha=\sum_{s\in {\supp(\alpha)}} n_a[a]$ to $\epsilon(\alpha)= \sum_{s\in \supp(\alpha)} n_a.$

\begin{thm}
Let $f\in \Q(x)\setminus \Q$. Write
\[
f=f_+ + f_-,
\qquad
f_+(x)=\frac{f(x)+f(1/x)}2,\quad
f_-(x)=\frac{f(x)-f(1/x)}2.
\]
 Then one of the following cases occurs:
\begin{enumerate}
    \item If $f_+=0$, then     \(    \Ann_{\Z[\Z^\circ]}(f)=\langle [1]+[-1]\rangle.    \)

    \item If $f_+$ is a nonzero constant, then     \(     \Ann_{\Z[\Z^\circ]}(f)=\ker(\epsilon)\langle [1]+[-1]\rangle.    \)

    \item If $f_-=0$, then     \(     \Ann_{\mathbb Z[\mathbb Z^\circ]}(f)=\langle [1]-[-1]\rangle.    \)

    \item If $f_-$ is a nonzero constant, then     \(    \Ann_{\Z[\Z^\circ]}(f)=\ker(\epsilon)\langle [1]-[-1]\rangle.     \)

    \item If $f_+$ and $f_-$ are both nonconstant, then     \(    \Ann_{\Z[\Z^\circ]}(f)=0.    \)
\end{enumerate}
\end{thm}

\begin{proof}
    Let $\alpha\in \Z [\Z^\circ]$ and write \(\alpha=A+[-1]B,\) where \(A,B\in \mathbb Z[\mathbb N^\circ]\). As in the proof of Corollary A.5, we see that 
    $\alpha \in \Ann_{\Z[\Z^\circ]}(f) $ if and only if
    \[
    (A+B)f_+=(A-B)f_- =0.
    \]

    (1) Suppose $f_+=0$. This implies that $f_-$ is nonconstant. By Lemma A.4,  the condition $(A-B)f_-=0$ is equivalent to $A-B=0$, i.e, $\alpha=A([1]+[-1])$. Hence \( \Ann_{\Z[\Z^\circ]}(f)=\langle [1]+[-1]\rangle.  \)

    (2) Suppose $f_+=c$, where $c$ is a nonzero constant. This implies that $f_-$ is nonconstant. By Lemma A.4,  the condition $(A-B)f_-=0$ is equivalent to $A-B=0$. And the condition $0=(A+B)f_+=2A(c)=2c\epsilon(A)$ is equivalent to $A\in \ker\epsilon$. Hence \( \Ann_{\Z[\Z^\circ]}(f)=\ker(\epsilon)\langle [1]+[-1]\rangle.  \)
    
    Case (3) and Case (4) can be proved by similar arguments as in case (1) and case (2).

    (5) Suppose $f_+$ and $f_-$ are both nonconstant.  By Lemma A.4,  the condition $(A+B)f+=(A-B)f_- = 0$ is equivalent to $A+B=A-B=0$. This latter condition is equivalent to $A=B=0$. Hence   \(    \Ann_{\Z[\Z^\circ]}(f)=0.    \)
    
\end{proof}

%\section{Abstract version of Section 4}

%\bibliographystyle{amsplain}

%\bibliography{references.bib}

\end{document}